\documentclass[12pt,reqno]{article}
\usepackage{epsfig,amssymb,latexsym,amsmath,pifont,multicol,mathtools,amsmath,verbatim,amsthm,float,lipsum}
\usepackage{caption} 
\usepackage[utf8]{inputenc}
\usepackage[T1]{fontenc}

\usepackage[varg]{txfonts}
\let\mathbb=\varmathbb

\usepackage[sans]{dsfont}

\usepackage[left=2cm, right=2cm, top=2cm, bottom=2cm]{geometry}

\usepackage[%
cal=euler,
bb=fourier,
scr=euler,
]
{mathalfa}

\usepackage[font=small,labelfont=bf]{caption}
\usepackage{subfigure}
\usepackage{graphicx}
\graphicspath{{Figures/}} 
\usepackage{acronym}
\usepackage{latexsym}
\usepackage{paralist}
\usepackage{wasysym}
\usepackage{xspace}
\usepackage{framed}
\usepackage{palatino,pxfonts}
\usepackage{authblk}
\usepackage{enumitem}
\usepackage[numbers,sort&compress]{natbib}

\usepackage[dvipsnames,svgnames]{xcolor}
\colorlet{MyBlue}{DodgerBlue!75!Black}
\colorlet{MyGreen}{DarkGreen!95!Black}

\newcommand{\SQ}[1]{{\color{Black}#1}}
\newcommand{\revise}[1]{{\color{Black}#1}}
\usepackage{hyperref}
\hypersetup{
colorlinks=true,
linktocpage=true,
pdfstartview=FitH,
breaklinks=true,
pdfpagemode=UseNone,
pageanchor=true,
pdfpagemode=UseOutlines,
plainpages=false,
bookmarksnumbered,
bookmarksopen=false,
bookmarksopenlevel=1,
hypertexnames=true,
pdfhighlight=/O,
urlcolor=MyBlue!60!black,linkcolor=MyBlue!70!black,citecolor=DarkGreen!70!black, % <--- for screen
pdftitle={},
pdfauthor={},
pdfsubject={},
pdfkeywords={},
pdfcreator={pdfLaTeX},
pdfproducer={LaTeX with hyperref}
}
\numberwithin{equation}{section}  %numberwithin goes before cleverefs when using hyperref
\usepackage[sort&compress,capitalize,nameinlink]{cleveref}
\crefname{example}{Ex.}{Exs.}

\crefrangeformat{equation}{\upshape(#3#1#4)\textendash(#5#2#6)}

\usepackage{mathtools}
\usepackage[para,online,flushleft]{threeparttable}
\usepackage{multirow}

\usepackage{physics}
\usepackage{float}
\DeclareMathOperator*{\argmin}{argmin}

\DeclareMathOperator{\bd}{bd}
\DeclareMathOperator{\cl}{cl}

\DeclareMathOperator{\zer}{zer}

\DeclareMathOperator{\dist}{dist}
\DeclareMathOperator{\dif}{d\!}

\DeclareMathOperator{\dom}{dom}

\DeclareMathOperator{\gr}{gr}
\DeclareMathOperator{\Int}{int}

\DeclareMathOperator{\range}{Range}
\DeclareMathOperator{\Avg}{Avg}

\DeclareMathOperator{\Id}{Id}

\DeclareMathOperator{\indicator}{\iota}

\newcommand{\ca}{\mathtt{a}}
\newcommand{\cb}{\mathtt{b}}
\newcommand{\ce}{\mathtt{e}}

\newcommand{\B}{\mathbb{B}}

\newcommand{\bx}{\mathbf{x}}

\newcommand{\bS}{\mathbf{S}}

\renewcommand{\iff}{\Leftrightarrow}

\renewcommand{\emptyset}{\varnothing}
\newcommand{\eqdef}{\triangleq}

\newcommand{\scrZ}{\mathcal{Z}}

\newcommand{\scrB}{\mathcal{B}}
\newcommand{\scrC}{\mathcal{C}}
\newcommand{\scrD}{\mathcal{D}}
\newcommand{\scrE}{\mathcal{E}}
\newcommand{\scrF}{\mathcal{F}}
\newcommand{\scrG}{\mathcal{G}}
\newcommand{\scrH}{\mathcal{H}}

\newcommand{\scrS}{\mathcal{S}}

\newcommand{\setS}{\mathscr{S}}

\renewcommand{\Pr}{\mathbb{P}}
\newcommand{\Ex}{\mathbb{E}}

\newcommand{\1}{\mathbf{1}}

\newcommand{\R}{\mathbb{R}}

\newcommand{\N}{\mathbb{N}}

\newcommand{\K}{\mathbb{K}}
\DeclareMathOperator{\NC}{\mathrm{N}}

\newcommand{\bC}{{\mathbf{C}}}

\newcommand{\BV}{\mathbf{BV}}
\newcommand{\Lim}{\mathsf{Lim}}
\newcommand{\ball}{\mathbb{B}}
\newcommand{\opA}{\mathsf{A}}

\newcommand{\opD}{\mathsf{D}}
\newcommand{\opM}{\mathsf{M}}

\DeclareMathOperator{\gap}{\Theta}

\DeclareMathOperator{\support}{\mathsf{s}}
\usepackage[ruled,vlined]{algorithm2e}
\usepackage{algpseudocode}								% for algorithm macros
\theoremstyle{plain}
\newtheorem{theorem}{Theorem}
\newtheorem{corollary}[theorem]{Corollary}
\newtheorem*{corollary*}{Corollary}
\newtheorem{lemma}[theorem]{Lemma}
\newtheorem{proposition}[theorem]{Proposition}

\theoremstyle{definition}
\newtheorem{definition}[theorem]{Definition}
\newtheorem*{definition*}{Definition}
\newtheorem*{problem*}{Problem}
\newtheorem{assumption}{Assumption}
\renewcommand{\qed}{\hfill{\footnotesize$\blacksquare$}}

\theoremstyle{remark}
\newtheorem{remark}{Remark}
\newtheorem*{remark*}{Remark}
\newtheorem*{notation*}{Notational remark}

\numberwithin{theorem}{section}
\numberwithin{remark}{section}
\numberwithin{example}{section}

\DeclarePairedDelimiter{\inner}{\langle}{\rangle}

\title{Asymptotic behaviour of coupled random dynamical systems with multiscale aspects}
\date{\today}

\author[3]{\small D. Russell Luke} 
\author[1]{\small Johannes-Carl Schnebel}
\author[4]{\small Mathias Staudigl}
\author[2]{\small Juan Peypouquet} 
\author[1,2]{\small Siqi Qu}

\affil[1]{\footnotesize Mannheim University, Department of Mathematics, B6 26, 68159 Mannheim, Germany\\
(\href{mailto: qu.siqi@uni-mannheim.de}{qu.siqi@uni-mannheim.de},\href{mailto: j.schnebel@uni-mannheim.de}{j.schnebel@uni-mannheim.de})}

\affil[2]{\footnotesize Rijksuniversiteit Groningen, Faculty of Science and Engineering, Systems, Control and Optimization — Bernoulli Institute, Groningen, The Netherlands\\
(\href{mailto:j.g.peypouquet@rug.nl}{j.g.peypouquet@rug.nl})}

\affil[3]{\footnotesize Institut für Numerische und Angewandte Mathematik, Universität Göttingen, 37083 Göttingen, Germany \\
(\href{mailto: r.luke@math.uni-goettingen.de}{r.luke@math.uni-goettingen.de})}

\affil[4]{\footnotesize Department of Business analytics and Decision sciences, Vienna University of Economics and Business, Welthandelsplatz 1, 1020 Wien, Austria \\
(\href{mailto: mathias.staudigl@wu.ac.at}{mathias.staudigl@wu.ac.at})}

\begin{document}

\maketitle

\begin{abstract}
We examine a class of stochastic differential inclusions involving multiscale effects designed to solve a class of generalized variational inequalities. This class of problems contains constrained convex non-smooth optimization problems, constrained saddle-point problems and various equilibrium problems in economics and engineering. In order to respect constraints we adopt a penalty approach, introducing an explicit time-dependency into the evolution system. The resulting dynamics are described in terms of a non-autonomous stochastic evolution equation governed by maximally monotone operators in the drift and perturbed by a Brownian motion. We study the asymptotic behavior, as well as finite time convergence rates in terms of gap functions. The condition we use to prove convergence involves a Legendre transform of the function describing the set constrained domain, a condition first used by Attouch and Czarnecki (J. Differ. Equations, Vol. 248, Issue 6, 2010) in the context of deterministic evolution equations. We also establish a large deviations principle showing that individual trajectories exhibit exponential concentration around the solution set. Finally we show how our continuous-time approach relates to penalty-regulated algorithms of forward-backward type after performing a suitable Euler-Maruyama discretization.
\end{abstract}

\section{Introduction}
\label{sec:intro}
%----------------------------------------------------------------------
%%%Intro 
%----------------------------------------------------------------------
%!TEX root = ./MultiscaleStochastic.tex
%

In this paper we are concerned with the study of stochastic splitting methods for solving generalized variational inequalities of the form 
\begin{equation}\label{eq:MI}\tag{MI}
\text{Find $\bar{x}\in C$ such that }0\in(\opA+\NC_{C})(\bar{x}),
\end{equation}
where the problem data satisfy the following standing assumptions:
\begin{assumption}\label{ass:OpMonotone}
\begin{itemize}
\item $\opA:\R^{d}\to 2^{\R^{d}}$ is a maximally monotone operator and $C\subset\R^{d}$ is a non-empty closed convex set. $\NC_{C}$ is the outward normal cone mapping to $C$.
\item $\dom(\opA)=\{x\in\R^{d}\vert\opA(x)\neq\emptyset\}$ has a nonempty interior: $\Int(\dom\opA)\neq\emptyset$.
\item \SQ{$\opA+\NC_{C}$} is maximally monotone.
\item $\setS=\{x\in\R^{d}\vert 0\in \opA(x)+\NC_{C}(x) \}\neq\emptyset$.
\end{itemize}
\end{assumption}
Maximal monotonicity means that \eqref{eq:MI} is a monotone inclusion problem which requires finding $\bar{x}\in C$ and $\bar{x}^{*}\in\opA(\bar{x})$ such that $\inner{\bar{x}^{*},x-\bar{x}}\geq 0$ for all $x\in C$.\\
When $\opA=\partial\Phi$ for $\Phi:\R^{d}\to(-\infty,\infty]$, a proper  and lower semi-continuous convex function, then the problem \eqref{eq:MI} reduces to the optimality conditions for the constrained convex minimization problem
\begin{equation}\label{eq:Opt}\tag{Opt}
\Phi_{\rm Opt}(C)=\min\{\Phi(x)\vert x\in C\}.
\end{equation}
If $\opA=F$, where $F:\R^{d}\to\R^{d}$ is a single-valued Lipschitz continuous and monotone mapping, the monotone inclusion problem \eqref{eq:MI} reduces to the variational inequality 
\begin{equation}\label{eq:VI}\tag{VI} 
\text{Find $\bar{x}\in C$ such that }\inner{F(\bar{x}),x-\bar{x}}\geq 0\qquad\text{\SQ{for all} }  x\in C.
\end{equation}
This general equilibrium framework includes various game theoretic equilibrium concepts, as well as min-max optimization problems. 

The question how to respect state constraints in a dynamical system is always an important and difficult question in scientific computing. While direct projection methods onto the feasible set $C$ would be possible with noisy oracles (leading to the projected stochastic subgradient method \cite{beck2017first}), computing a projection is very often a computationally challenging task. Penalty methods are a natural tool to steer iterative methods towards the satisfaction of the set constraint $C$. We thus assume that the set $C$ admits a representation in terms of the level set of a continuously differentiable function $\Psi:\R^{d}\to\R$ satisfying 
\begin{assumption}\label{ass:Psi}
 $\Psi:\R^{d}\to\R$ is a convex and continuously differentiable function, whose gradient mapping $x\mapsto \nabla\Psi$ is \SQ{$L_{\Psi}$-}Lipschitz continuous \SQ{for some $L_{\Psi}>0$}. Moreover, $C=\Psi^{-1}(0)=\argmin\Psi\neq\emptyset$.
\end{assumption}
Our approach in this paper is guided by the systematic study of dissipative dynamical systems which have been developed over the last decades in order to solve problem \eqref{eq:MI}. Indeed, the early references \cite{attouch2010asymptotic,attouch2018asymptotic} studied the asymptotic behavior of the differential inclusion 
\begin{equation}\label{eq:DI}\tag{DI}
\dif \bx(t)+\opA(\bx(t))\dif t+\beta(t)\nabla \Psi(\bx(t))\dif t \ni 0 
\end{equation}
or, as a special case of the above with $\opA=\partial\Phi$ for a convex proper and lower semi-continuous function $\Phi:\R^{d}\to[0,\infty]$, 
\begin{equation}\label{eq:MAG}\tag{MAG}
\dif \bx(t)+\partial \Phi(\bx(t))\dif t+\beta(t)\nabla \Psi(\bx(t))\dif t \ni 0 
\end{equation}
Time dependency is included in these dynamical systems via an absolutely continuous function $t\mapsto\beta(t)$, modeled as a positive and non-decreasing penalty parameter. To understand its role, in the setting of \eqref{eq:MAG} observe that $\partial\Phi+\beta(t)\nabla\Psi\subset\partial(\Phi+\beta(t)\Psi),$ so that each trajectory of \eqref{eq:MAG} satisfies 
$$
0\in \dif \bx(t)+\partial(\Phi+\beta(t)\Psi)(\bx(t)).
$$
On the other hand, $\Phi+\beta(t)\Psi\uparrow\Phi+\indicator_{C}$ as $t\to\infty$, \SQ{where $\iota_C$ denotes the indicator function of $C$}. Since monotone convergence is a variational convergence \cite{attouch1984variational}, we have also $\partial(\Phi+\beta(t)\Psi)\to\partial(\Phi+\indicator_{C})$, so that the dynamical system in the previous display reduces to the subgradient flow 
$$
0\in\dif\bx(t)+\partial(\Phi+\indicator_{C})(\bx(t)).
$$
It has been shown in the above works that - under certain conditions on the function $\beta(\cdot)$ - solutions of these dynamical systems asymptotically approach equilibria to \eqref{eq:MI}, respectively \eqref{eq:Opt}. 

In many cases the evaluation of the involved operators is subject to noise, either due to inexact evaluations, or because relevant problem data are only available in terms of a stochastic oracle. Motivated by data-driven approaches to control and optimization, we work in settings in which information about the variational inequality cannot be evaluated directly, but rather only subject to random exogenous noise. In such scenarios, it has frequently been argued in machine learning that these errors can be modeled in terms of a stochastic integral \cite{li2019stochastic,an2020stochastic}. This is the starting point for our work. Following continuous-time approaches \cite{MerStauSIAM,MerStaJOTA18,raginsky2012continuous,rodrigo2024stochastic,Maulen-Soto:2025aa}, we assume that we can model the associated errors in terms of a stochastic integral leading to stochastic differential inclusions (SDI) governed by maximally monotone operators of the form
\begin{equation}\label{eq:SDIMMO}
\left\{
\begin{aligned}
&\dif X(t)+\opA(X(t))\dif t+\beta(t)\nabla\Psi(X(t))\dif t\ni \sigma(t,X(t))\dif W(t)\\
 &X(0)=x\in\cl(\dom(\opA)), 
 \end{aligned}\right.
\end{equation}
a stochastic process on a stochastic basis $(\Omega,(\scrF_{t})_{t\geq 0},\scrF,\Pr)$, carrying \SQ{an} $m$-dimensional Brownian motion $\{(W(t),\scrF_{t});t\geq 0\}$. The diffusion (volatility) term $\sigma(t,x):[t_{0},\infty)\times\R^{d}\to \R^{d\times m}$ is a matrix-valued measurable function. For a discussion on \SQ{the} existence and uniqueness of solution, see Appendix \ref{app:existence}.

As a special case of this dynamical approach, we derive results for the constrained convex minimization problem, defined by the maximally monotone operator $\opA=\partial\Phi$, in terms of the set-valued stochastic dynamical system
\begin{equation}\label{eq:SDI}
\left\{
\begin{aligned}
&\dif X(t)+\partial \Phi(X(t))\dif t +\beta(t)\nabla\Psi(X(t))\dif t\ni \sigma(t,X(t))\dif W(t), \\ 
&X(0)=x\in\cl(\dom(\partial \Phi)).
\end{aligned}\right.
\end{equation}
The set-valued formulation aims for the design of a continuous stochastic process $X$, whose sample paths are confined to stay in $\dom(\opA)$. Our analysis can be seen as a direct stochastic extension of the study of non-autonomous evolution equations from the lens of convex and variational analysis, inspired by \cite{attouch2010asymptotic}. The multiscale character of our dynamical system is induced by the increasing schedule for the penalty parameter $\beta(t)$, so that the trajectories are forced over time to enter the feasible set $C$. Within this formulation, we prove new results on the asymptotic and finite-time properties of sample paths induced by the random dynamical systems \eqref{eq:SDIMMO} and \eqref{eq:SDI}, respectively. Specifically, the main results of this paper are the following: 
\begin{itemize}
\item In the setting of the SDI \eqref{eq:SDIMMO}, we prove $L^{p}$ estimates on the trajectories in terms of the geometry of the  exterior penalty function $\Psi$, which has been already systematically exploited in the study of deterministic evolution equations in \cite{attouch2018asymptotic} and \cite{Bot:2016aa}. The algorithmic implications of this geometric setting were first studied in \cite{attouch2011prox} (see also \cite{AttCzarPey11,Peypouquet:2012aa,noun2013forward,Bot:2014aa,czarnecki2016splitting}), although they can be traced back to the early work of Cabot \cite{cabot2005proximal}, whose condition involves both functions $\Phi$ and $\Psi$, in a hierarchical optimization setting. \revise{However, the presence of random perturbations requires different proof techniques compared to the deterministic analysis. Indeed, our arguments rely on various tools from stochastic and variational analysis so that the proofs are significantly different compared to the main results reported in \cite{attouch2010asymptotic}.}  
\item We establish convergence rates to the solution set $\scrS$ of the variational inequality \eqref{eq:MI} in terms of gap functions, and prove almost sure convergence of the ergodic average to a random variable taking values in $\scrS$, under a weak sharpness condition on $\scrS$. \revise{This is a new result compared to \cite{attouch2010asymptotic}, adding a quantitative estimate to the asymptotic convergence result.}
\item Specializing to the potential case $\opA=\partial\Phi$, we establish asymptotic convergence and finite time rates in terms of the objective function gap. Note that in this case, the dynamical system \eqref{eq:SDI} can be seen as a stochastic method for solving the hierarchical minimization problem 
\begin{equation}\label{eq:SBP}
\min\{\Phi(x)\vert x\in\argmin_{u\in\R^{d}}\Psi(u)\}
\end{equation}
This is also known as the simple bilevel optimization problem, a challenging class of non-convex optimization problems which received a lot of attention over the last 10 years \cite{sabach2017first,Dempe:2021aa}.
\item Building on \cite{Bernardin:2003aa}, we relate the continuous-time dynamical system to a discrete-time algorithm of forward-backward type, for which we prove new complexity results. As a corollary of our investigations, we derive the first complexity results for a stochastic forward-backward type algorithm devised to solve hierarchical minimization problems of the form \eqref{eq:SBP}.
\end{itemize}
 
\subsection{Related literature} 
Our approach draws inspiration from variational analysis and probability theory. The long-run properties of such dynamics have been studied intensively in convex optimization \cite{attouch2011prox,AttCzarPey11,Peypouquet:2012aa,noun2013forward,czarnecki2016splitting} in a noise-free setting. \revise{Specifically, motivated by the development of splitting methods to resolve the constrained variational inequality problem \eqref{eq:MI}, \cite{AttCzarPey11} develop diagonal algorithms motivated by the non-autonomous differential inclusion \eqref{eq:DI}, and prove the asymptotic convergence of the discrete-time scheme. Following these publications, the references \cite{Bot:2016aa,boct2014forward,banert2015backward} study the more general inclusion problem 
\begin{equation}\label{eq:MI2}
0\in \opA(x)+\NC_{C}(x)+\opD(x)
    \end{equation}
with $\opD$ a monotone and Lipschitz continuous operator. Penalty-regulated dynamical systems for resolving \eqref{eq:MI2}, including Tikhonov regularization and penalty effects have been studied in \cite{10530399,peypouquet2026asymptotic}.} Only recently, the stochastic analysis of evolution systems of the above kind has been started. In the smooth case, \cite{raginsky2012continuous,MerStaJOTA18,MerStauSIAM} studied projection-based SDE models for convex optimization in the mirror descent framework. Recently, \cite{rodrigo2024stochastic} studied proximal gradient dynamics subject to a Brownian motion. \cite{Bot:2025aa} studies random dynamical systems for operator equations. Finally, \cite{maulen2024tikhonov} develops stochastic differential inclusions with Tikhonov regularization. They use a quite restrictive definition of a solution, requiring that the Yosida approximation of the maximally monotone operator obey a linear growth condition. This requires that the set-valued operator $\opA$ has full domain, a rather restrictive assumption for multivalued stochastic differential inclusions (see \cite{pettersson1995yosida}). We treat a more general class of solutions, allowing for discontinuous repulsive effects at the boundary of $\dom\opA$, at the price of requiring the interior of $\dom\opA$ is nonempty. This is a fairly standard assumption in the literature on stochastic differential inclusions governed by maximally monotone operators \cite{pardoux2014stochastic,barbu2020optimal,BarbuPrato08,barbu2005neumann}. Indeed, stochastic differential inclusions governed by maximally monotone operators have been systematically investigated over the last decades. Main results on existence and uniqueness of solutions have been reported in finite dimensions in \cite{cepa2006equations} and in infinite dimensions in \cite{Rascanu96}. Solutions to this system comprise a very large class of probabilistic problems, including reflected diffusions, the generalized Skorokhod problem \cite{MENALDI:1983aa,pettersson1995yosida} and the analysis of non-linear Fokker-Planck equations \cite{Ciotir:2022aa,barbu2020optimal}. The textbook \cite{pardoux2014stochastic} gives a good introduction to this field. Our main motivation for studying set-valued dynamical systems governed by maximally monotone operators is to solve monotone inclusions of the form \eqref{eq:MI}. 

%\begin{equation}\label{eq:Skorohod}
%\dif X(t)+\partial\indicator_{\cl(O)}(X(t))(\dif t)\ni \dif m(t) \quad X(0)=x_{0}.
%\end{equation}
%where $\setO\subset\R^{d}$ is a convex domain with smooth boundary $\bd(O)$ and non-empty interior, and $\indicator_{\cl(O)}$ is the convex indicator function of its closure. From convex analysis, we know that $%\partial\indicator_{\cl(\setO)}(x)=\NC_{\cl(O)}(x)$, so that problem is equivalent to finding two functions $(x,k):\R_{\geq 0}\to\R^{d}$ such that 
%\begin{align*}
%&x\in\bC(\R_{\geq 0};\R^{d}),x(0)=x_{0},x(t)\in\cl(O)\text{ for all }t\geq 0 \\ 
%&k\in\bC(\R_{\geq 0};\R^{d})\cap\BV(\R_{\geq 0};\R^{d}),k(0)=0\\
%&x(t)+k(t)=x_{0}+m(t)\quad \forall t\geq 0 \\ 
%&\int_{0}^{t}\inner{y-x(r),\dif k(r)}\leq 0 \qquad\forall y\in\cl(O). 
%\end{align*}
%The last condition means that $\dif k(t)\in\partial\indicator_{\cl(O)}(x(t))(\dif t)$; see \cite{cepa2006equations,pardoux2014stochastic}. Hence, one can rewrite this problem as 
%$$
%\dif x(t)+\partial\indicator_{\cl(O)}(x(t))(\dif t)\ni \dif m(t) ,\quad x(0)=x_{0}.
%$$

\paragraph{Organization of the paper}
We provide a complete analysis of the strong solutions of the system \eqref{eq:SDIMMO}, in Section \ref{sec:General}. Besides asymptotic convergence towards solutions of the constrained variational inequality \eqref{eq:MI}, we develop a rate estimate of the time averaged trajectory in terms of a suitably defined gap function. Section \ref{sec:subgradient} gives refined estimates on the trajectory for the differential inclusion \eqref{eq:SDI}. Section \ref{sec:discrete} \SQ{establishes} interesting connections between the continuous-time system \eqref{eq:SDIMMO} and a discrete-time stochastic forward-backward splitting method, first studied in \cite{attouch2011prox,AttCzarPey11} in the deterministic setting. Section \ref{sec:numerics} reports numerical experiments obtained from our numerical scheme. Section \ref{sec:conclusion} concludes the paper with further connections to the existing literature and promising future directions.

\section{Notation and preliminaries}
\label{sec:prelims}
%----------------------------------------------------------------------
%%%Prelims 
%----------------------------------------------------------------------
%!TEX root = ./MultiscaleStochastic.tex
%

\paragraph{Elements of variational analysis}
For a matrix $A\in\R^{d\times k}$ we define its Frobenius norm as $\tr(A^{\top}A)=\norm{A}^{2}_{\rm F}$. Let $\Gamma_{0}(\scrH)$ denote the set of all proper lower semi-continuous convex functions on a Hilbert space $\scrH$. The norm and inner product in $\scrH$ are denoted by $\norm{\cdot}$ and $\inner{\cdot,\cdot}$, respectively. Given $f\in\Gamma_{0}(\scrH)$ and $x\in\scrH$, the subdifferential of $f$ at $x$ is the set $\partial f(x)=\{x^{*}\in\scrH\vert f(y)\geq f(x)+\inner{x^{*},y-x}\;\forall y\in\scrH\}$. The Fenchel conjugate of $f\in\Gamma_{0}(\scrH)$ is defined as 
$f^{*}(x^{*})=\sup_{x\in\scrH}\{\inner{x^{*},x}-f(x)\}$. Given a nonempty closed convex set $C\subset\scrH$, its indicator function is defined as $\indicator_{C}(x)=0$ if $x\in C$ and $+\infty$ otherwise. The support function of $C$ is $\support_{C}(x^{*})=\sup_{u\in C}\inner{x^{*},u}$. The normal cone to $C$ is 
$$
\NC_{C}(x)=\left\{\begin{array}{ll} \emptyset  & \text{ if }x\notin C,\\ 
\{p\in\R^{d}\vert \inner{y-x,p}\leq 0 \quad\forall y\in C\} & \text{else.}
\end{array}\right.
$$
It is clear that $\partial\indicator_{C}(x)=\NC_{C}(x)$. 

\SQ{For $L>0$, we denote by $C^{1,1}_{L}(\R^d)$ the space of continuously differentiable functions on $\R^d$ whose gradients are $L$-Lipschitz continuous.}

A monotone operator is a set-valued mapping $\opM:\scrH\to 2^{\scrH}$ such that 
\[
\inner{x^{*}-y^{*},x-y}\geq 0\text{ whenever }(x,x^{*})\in\gr(\opM),(y,y^{*})\in\gr(\opM), 
\]
where $\gr(\opM)=\{(x,x^{*})\vert x^{*}\in\opM\}$ is the graph of $\opM$. If $\opM=\opA+\NC_{C}$ is maximal monotone and $\scrS=\opM^{-1}(0)$, then $z\in\scrS$ if and only if
%\begin{equation}\label{eq:eqi}
%(\forall (u,w)\in\gr(\opA+\NC_{C})): \; \inner{w,u-z}\geq 0.
%\end{equation}
\SQ{
\begin{equation}\label{eq:eqi}
\inner{w,u-z}\geq 0
\quad\text{for all }(u,w)\in\gr(\opA+\NC_{C}).
\end{equation}
}

\paragraph{Elements of probability theory}
Let $(\Omega,\scrF,\Pr)$ \SQ{be} a probability space and $\{\scrF_{t};t\geq 0\}$ \SQ{be} a filtration. A $\scrH$-valued stochastic process starting at $t_{0}\geq 0$ is a function $X:\Omega\times[t_{0},\infty)\to\scrH$. If $X(\omega,\bullet)\in\bC([t_{0},\infty);\scrH)$, then it is called continuous. It is called progressively measurable if for every $t\geq t_{0}$, the restriction $X\vert_{\Omega\times[t_{0},t]}$ is $\scrF_{t}\otimes\scrB([t_{0},t])$-measurable. The expectation on a measurable set $A\subset\Omega$ is defined as $\Ex[X;A]=\int_{A}X\dif\Pr(\omega)=\Ex(X\mathds{1}_{A})$. If $A=\Omega$, then we obtain the expected value of the random variable $X$.

We define the quotient space $\bS^{0}_{d}[t_{0},T]$ as the space of equivalence classes of progressively measurable processes $X:\Omega\times[t_{0},T]\to\R^{d}$. We set $\bS_{d}^{0}[t_{0}]\eqdef \bigcap_{T\geq t_{0}}\bS^{0}_{\R^{d}}[t_{0},T]$. For $p>0$, we define $\bS^{p}_{d}[t_{0},T]$ the space of equivalence classes of progressively measurable processes $X:\Omega\times[t_{0},T]\to\R^{d}$ such that 
$\Ex\left(\sup_{t\in[t_{0},T]}\norm{X(t)}^{p}\right)<\infty.$

%\color{cyan}
%Set
%     $$\bar X_n(\omega):=\frac{1}{n}\sum_{k=1}^nX_k(\omega).$$
%%We have the following:
%\begin{proposition}\label{prop:CP}
%Let $\scrS$ be a nonempty closed subset of $\R^{d}$ and $(X_{n})_{n}$ a stochastic process living in $\R^{d}$. Suppose that there exists $\Omega_{0}\subset\Omega$ with $\Pr(\Omega_{0})=1$ and, for every $\omega\in\Omega_{0}$, and every $z\in\scrS$, the process $(\norm{X_{n}(\omega)-z})_{n\in\N}$ converges. If additionally all accumulation points of $(\bar X_{n}%(\omega))_{n\in\N},\omega\in\Omega_{0},$ are contained in $\scrS$, then $(\bar X_{n})_{n\in\N}$ converges $\Pr$-a.s. to a $\scrS$-valued random variable.%
%\end{proposition}
%%
 %%}

\section{Stochastic evolution equations governed by maximally monotone operators}
\label{sec:General}
%----------------------------------------------------------------------
%%%OperatorEvolution 
%----------------------------------------------------------------------
%!TEX root = ./MultiscaleStochastic.tex
%
%\subsection{Strong solutions to \eqref{eq:SDIMMO}}
\subsection{Strong solutions to \texorpdfstring{\eqref{eq:SDIMMO}}{the stochastic differential inclusion}}
\SQ{We first give a precise meaning to the concept of solution to} the random set-valued dynamical system \eqref{eq:SDIMMO}. We are given a stochastic basis $(\Omega,\scrF,(\scrF_{t})_{t},\Pr)$, carrying \SQ{an} $m$-dimensional Brownian motion $W=\{(W(t),\scrF_{t}),0\leq t<\infty\}$. 
\begin{definition}\label{def:SDIOperator}
Let $\xi\in L^{2}(\Omega;\cl(\dom(\opA)))$ be a given random variable. A pair of continuous $\R^{d}$-valued processes $X=\{(X(t),K(t)),t\in [0,T]\}$ is called a \emph{strong solution} to \eqref{eq:SDIMMO} if the following conditions hold. 
\begin{enumerate}
\item[(SE.1)] $(X(t))_{t\in[0,T]}$ is $\scrF_{t}$-adapted with continuous sample paths such that $X(t)\in\cl(\dom(\opA))$ $\Pr$-a.s. for all $t\in[0,T]$; 
\item[(SE.2)] $X(0)=\xi\quad\Pr$-a.s.
\item[(SE.3)] $(K(t))_{t\in[0,T]}$ is \SQ{an} $\R^{d}$-valued process, $\scrF_{t}$-adapted whose sample paths are continuous and of bounded variation on $[0,T]$ such that 
%$$
%\dif X(t)=-\beta(t)\nabla\Psi(X(t))\dif t+\sigma(t,X(t))\dif W(t)-\dif K(t)
%$$
%holds a.e. $t\in[0,T]$ and $\Pr$-a.s. Here $\dif K(t)$ is interpreted as the Lebesgue-Stieltjes measure generated by $K$.
\revise{
$$
X(t)=\xi-\int_0^t \beta(s)\nabla\Psi(X(s))\,\dif s+\int_0^t \sigma(s,X(s))\,\dif W(s)-K(t)+K(0)
$$
for all $t\in[0,T]$, $\Pr$-a.s.} Here, $\dif K(t)$ is interpreted as the Lebesgue--Stieltjes measure generated by $K$.
\item[(SE.4)] For all $\R^{d}$-valued processes $(\zeta,\zeta^{\ast})$ with $\zeta=\{(\zeta(t),\scrF_{t});t\geq 0\}$ and $\zeta^{\ast}=\{(\zeta^{\ast}(t),\scrF_{t});t\geq 0\}$ continuous, adapted on $(\Omega,\scrF,\Pr)$, verifying $(\zeta(t),\zeta^{\ast}(t))\in\gr(\opA)$ for all $t\in[0,T]$, the measure 
$$
\inner{X(t)-\zeta(t),\dif K(t)-\zeta^{\ast}(t)\dif t}
$$
is \SQ{$\Pr$-a.s.} non-negative on $[0,T]$. 
\end{enumerate}
\end{definition}
\begin{remark} 
This last condition justifies the suggestive notation $\dif K(t)\in \opA(X(t))\dif t$.
\end{remark}

Our convergence results are derived within a specific geometric setting defined in terms of a summability condition on the Fenchel transform \SQ{of} the exterior penalty term $\Psi$. 

\revise{
\begin{assumption}[Attouch-Czarnecki condition, \cite{attouch2010asymptotic}]
\label{ass:CZ}
For all $p\in\range(\NC_{C})$ it holds 
\begin{equation}\label{eq:Fitz}
\int_{0}^{\infty}\beta(r)\left[\Psi^{\ast}\left(\frac{p}{\beta(r)}\right)-\support_{C}\left(\frac{p}{\beta(r)}\right)\right]\dif r<\infty.
\end{equation}
\end{assumption}}
While \eqref{eq:Fitz} appears complicated, it is weaker than a Hölderian growth condition imposed on the penalty function $\Psi$, and can be verified in many concrete settings.

\begin{definition}[Error bound]
\label{def:EB}
A proper convex and lower semi-continuous function $f:\R^{d}\to\R\cup\{+\infty\}$ with $\argmin(f)\neq\emptyset$ satisfies a Hölderian growth condition with exponent $\rho\in(1,2]$ if, for some $\tau>0$, we have 
$$
\frac{\tau}{\rho}\dist(x,\argmin f)^{\rho}\leq f(x)-\min f\quad \text{\SQ{for all }} x\in\R^{d}.
$$
\end{definition}
Assuming that $\Psi$ satisfies a Hölderian growth condition around $\argmin \Psi=C$, we have 
$$
0\leq \Psi^{\ast}(z)-\support_{C}(z)\leq\tau^{1-\rho^{*}}\frac{1}{\rho^{*}}\norm{z}^{\rho^{*}}\quad \text{\SQ{for all }} z\in\R^{d},
$$
where $\rho^{*}$ is the dual conjugate exponent given by $\frac{1}{\rho}+\frac{1}{\rho^{*}}=1$. Therefore, the Hölderian growth condition implies \eqref{eq:Fitz} whenever $\beta(t)$ is a function satisfying the integrability condition 
$$
\int_{0}^{\infty}\tau^{1-\rho^{*}}\norm{p/\beta(r)}^{\rho^{*}}\dif t=\norm{p}^{\rho^{*}}\tau^{1-\rho^{*}}\int_{0}^{\infty}\abs{\beta(t)}^{-\rho^{*}}\dif t<\infty.
$$
In the important case where $\Psi(z)=\frac{1}{2}\dist(z,C)^{2}$, condition \eqref{eq:Fitz} simplifies to $\int_{0}^{\infty}\frac{1}{\beta(t)}\dif t<\infty$, which is satisfied by the polynomial regime $\beta(t)=(1+t)^{a}$, with $a>1$.

Besides assumptions on the geometry of the penalization process, we impose standard assumptions on the data generating the dynamical system \eqref{eq:SDIMMO}. 

\begin{assumption}\label{ass:Noisebound}
The diffusion term $\sigma:[0,\infty)\times\R^{d}\to\R^{d\times m}$ satisfies the Lipschitz condition 
\begin{equation}\label{eq:noiseLip}
\norm{\sigma(t,x)-\sigma(t,y)}_{\rm F}\leq \ell(t)\norm{x-y}\quad \text{\SQ{for all }} x,y\in\R^{d}
\end{equation}
for $\ell(\bullet)\in L^{2}(0,\infty;\R_{\geq 0})$, and one of the following conditions:
\begin{itemize}
\item[(i)] \emph{Uniformly bounded variance}:  There exists $\sigma_{\ast}>0$ satisfying 
\begin{equation}\label{eq:UBV}\tag{UBV}
\norm{\sigma(t,x)}_{\rm F}^{2}\leq \sigma^{2}_{\ast}.
\end{equation}
\item[(ii)] \emph{Asymptotically small variance}: 
\begin{equation}\label{eq:ASV}\tag{ASV}
\Sigma(t)\eqdef \sup_{x\in\R^{d}}\norm{\sigma(t,x)}_{\rm F}\in L^{2}([0,\infty);\R_{+}).
\end{equation}
\end{itemize}
\end{assumption}

%\begin{remark}\label{rem:noise}
%Conditions \eqref{eq:UBV} and \eqref{eq:ASV} can be combined to the requirement that there exists a measurable function $\Sigma(\cdot)\in L^{p}([0,\infty);\R_{\geq 0})$, with $p\in\{2,\infty\}$ such that $\Sigma(t)\geq \sup_{x\in\R^{d}}\norm{\sigma(t,x)}_{\rm F}.$ 
%\end{remark} 

If the diffusion term $\sigma$ remains bounded away from zero (a case which is possible under assumption \eqref{eq:UBV}), then it cannot be expected that solutions to the SDI \eqref{eq:SDIMMO} converge to solutions to \eqref{eq:MI}. In such a \emph{persistent noise} scenario, it is more relevant to study the ergodic properties of the stochastic process $X$ in terms of its infinitesimal generator. Existence and uniqueness of invariant measures for set-valued dynamical systems of the form \eqref{eq:SDIMMO} require stationary dynamics (i.e. $\beta(t)\equiv \bar{\beta}$ for $t\geq 0$), a setting which is not the main interest of this paper. Still, Section \ref{sec:conclusion} summarizes known results for this stationary setting, and points out interesting directions for future research. However, in this paper we mostly focus on the case where $\sigma$ is a vanishing diffusion term, as captured by condition \eqref{eq:ASV}. 
  
\subsection{Almost sure convergence of the trajectory}
\cite{cepa2006equations} proved that under the stated hypothesis there exists a unique solution $(X,K)$ to the SDI \eqref{eq:SDIMMO}. Estimates in the space $L^{p}$ for $p\geq 2$ on the sample paths have subsequently been exhibited in \cite{Bernardin:2003aa}. It is thus known that $X\in\bS^{p}_{d}[0,T]$ for all finite $T>0$. In particular,  $X$ is a semi-martingale. We use this fact to derive stronger bounds exploiting the additional structure provided by the geometric setting embodied by the Attouch-Czarnecki condition (Assumption \ref{ass:CZ}). 

\begin{proposition}\label{prop:1}
Let $X$ be the solution to \eqref{eq:SDIMMO} with $X(0)=\xi\in L^{p}(\Omega;\cl(\dom(\opA)))$ for some $p\geq 2$. Let Assumption \ref{ass:OpMonotone}-\ref{ass:Noisebound} hold true. Then, for every $z\in\zer(\opA+\NC_{C})$, we have:
\begin{itemize}
\item[(i)] If condition \eqref{eq:UBV} applies, then for all $T>0$ we have $\Ex\left[\sup_{0\leq t\leq T}\norm{X(t)}^{p}\right]<\infty$.\\
If condition \eqref{eq:ASV} holds, then $\Ex[\sup_{t\geq 0}\norm{X(t)}^{p}]<\infty.$ 
\item[(ii)] If condition \eqref{eq:ASV} applies, then $\lim_{t\to\infty}\norm{X(t)-z}$ exists almost surely. Moreover,
$$
\lim_{t\to\infty}\Ex\left(\norm{X(t)-z}\right)=\Ex\left(\lim_{t\to\infty}\norm{X(t)-z}\right),
$$
\item[(iii)] If condition \eqref{eq:ASV} holds, then 
$$
\int_{0}^{\infty}\beta(s)\Psi(X(s))\dif s<\infty\quad \Pr-\text{a.s.}
$$
In particular, $\liminf_{t\to\infty}\Psi(X(t))=0$ \SQ{$\Pr$-}a.s. 
\end{itemize}
\end{proposition}
\begin{proof}
(i) Consider the function $\ce_{z}(x)\eqdef\frac{1}{2}\norm{x-z}^{2}$. Let $(X,K)$ be a solution to \eqref{eq:SDIMMO} in the sense of Definition \ref{def:SDIOperator}. Then $X$ is almost surely continuous and the process $t\mapsto E_{z}(t,\omega)\eqdef \ce_{z}(X(t,\omega))$ is progressively measurable. It\^{o}'s formula yields 
\begin{equation}\label{eq:ItoOpen}
\begin{split}
\dif E_{z}(t)&=\inner{X(t)-z,-\beta(t)\nabla\Psi(X(t))}\dif t-\inner{X(t)-z,\dif K(t)}+\inner{\sigma(t,X(t))^{\top}(X(t)-z),\dif W(t)}\\
&+\frac{1}{2}\norm{\sigma(t,X(t))}^{2}_{\rm F}\dif t.
\end{split}
\end{equation}
Define $c(t)\eqdef\int_{0}^{t}\frac{1}{2}\norm{\sigma(s,X(s))}^{2}_{\rm F}\dif s$ and $M_{z}(t)\eqdef \int_{0}^{t}\inner{\sigma(s,X(s))^{\top}(X(s)-z),\dif W(s)}.$
%where $\Sigma(\cdot)$ is a measurable function satisfying Assumption \ref{ass:Noisebound} (with $p\in\{2,\infty\}$; cf. Remark \ref{rem:noise}). Here process $\{(c(t));t\geq 0\}$ majorizes the It\^{o}-correction term 
%$$
%c(t)\geq \frac{1}{2}\int_{0}^{t}\norm{\sigma(s,X(s))}^{2}_{\rm F} \qquad \text{a.e.}.
%$$
Using $z\in C$, the convex gradient inequality implies
$$
0=\Psi(z)\geq\Psi(X(t))+\inner{\nabla\Psi(X(t)),z-X(t)}.
$$
Moreover, if $z\in\zer(\opA+\NC_{C})$, there exists $p\in\NC_{C}(z)$ such that $-p\in\opA(z)$. By definition of a solution to \eqref{eq:SDI}, we have from \revise{(SE.4)} the inequality
$$
\int_{0}^{t}\inner{X(s)-z,\dif K(s)}\geq\int_{0}^{t}\inner{X(s)-z,-p}\dif s.
$$
Then, for all $t\in[0,T]$, we can continue with
\begin{align}
E_{z}(t)-E_{z}(0) & \leq -\int_{0}^{t}\beta(s)\inner{X(s)-z,\nabla\Psi(X(s))}\dif s-\int_{0}^{t}\inner{X(s)-z,\dif K(s)}+M_{z}(t)+c(t)\nonumber\\ 
& \leq -\int_{0}^{t}\beta(s)\Psi(X(s))\dif s +\int_{0}^{t}\inner{X(s)-z,p}\dif s+c(t)+M_{z}(t)\label{eq:Used}\\
&=\int_{0}^{t}\beta(s)\left[\inner{\frac{p}{\beta(s)},X(s)}-\Psi(X(s))-\inner{\frac{p}{\beta(s)},z}\right]\dif s+c(t)+M_{z}(t)\nonumber\\
&\leq \int_{0}^{t}\beta(s)\left[\Psi^{\ast}\left(\frac{p}{\beta(s)}\right)-\support_{C}\left(\frac{p}{\beta(s)}\right)\right]\dif s+c(t)+M_{z}(t).\nonumber
\end{align}
Next, we define 
\begin{align*}
&h_{C}(x^{*})\eqdef \Psi^{\ast}(x^{*})-\support_{C}(x^{*})\quad \forall x^{*}\in\dom(\Psi^{\ast}),\text{ and }H(t)\eqdef  \int_{0}^{t}\beta(s)h_{C}(p/\beta(s))\dif s.
\end{align*}
Hence, we can continue the previous estimate as 
$$
\norm{X(t)-z}^{2}\leq \norm{X(0)-z}^{2}+2H(t)+2M_{z}(t)+2c(t)\quad \text{\revise{for all }}t\in[0,T].
$$
The process $\{(M_{z}(t),\scrF_{t}):0\leq t\leq T\}$ is a martingale and by Assumption \ref{ass:CZ}, the process $t\mapsto H(t)$ is almost surely bounded. Additionally, $X\in \bS^{p}_{d}[0,T]$ for $p\geq 2$, so by using $(\sum_{i=1}^{n}x_{i})^{p}\leq n^{p-1}\sum_{i=1}^{n}x_{i}^{p}$, we obtain 
$$
\norm{X(t)-z}^{p}\leq 4^{\frac{p-2}{2}}\left[\norm{X(0)-z}^{p}+2^{p/2}H(t)^{p/2}+2^{\frac{p}{2}}c(t)^{p/2}+2^{\frac{p}{2}}\abs{M_{z}(t)}^{p/2}\right].
$$
Since $t\mapsto M_{z}(t)$ is a continuous local martingale, the Burkholder-Davis-Gundy inequality \cite{protter2005stochastic} yields a universal constant $C_{p/2}$, so that
\begin{align*}
\Ex\left[\sup_{t\in[0,T]}\abs{M_{z}(t)}^{\frac{p}{2}}\right]&=\Ex\left[\sup_{t\in[0,T]}\abs{\int_{0}^{t}\inner{\sigma(s,X(s))^{\top}(X(s)-z),\dif W(s)}}^{\frac{p}{2}}\right]\\
&\leq C_{p/2}\Ex\left[\left(\int_{0}^{T}\norm{\sigma(s,X(s))^{\top}(X(s)-z)}^{2}\dif s\right)^{\frac{p}{4}}\right].
\end{align*}
Using next the boundedness of the volatility function and Young's inequality $ab\leq\frac{a^{2}}{2K}+\frac{Kb^{2}}{2}$, we can further bound this expression as 
\begin{align*}
4^{\frac{p-2}{2}}2^{\frac{p}{2}}\Ex\left[\sup_{t\in[0,T]}\abs{M_{z}(t)}^{\frac{p}{2}}\right]&\leq4^{\frac{p-2}{2}}2^{\frac{p}{2}} C_{p/2}\Ex\left[\left(\int_{0}^{T}\norm{\sigma(t,x)}^{2}_{\rm F}\cdot\norm{X(t)-z}^{2}\dif t\right)^{\frac{p}{4}}\right]\\
&\leq 4^{\frac{p-2}{2}}2^{\frac{p}{2}}C_{1}\Ex\left[\sup_{t\in [0,T]}\norm{X(t)-z}^{\frac{p}{2}}\cdot\left(\int_{0}^{T}\norm{\sigma(t,x)}^{2}_{\rm F}\dif t\right)^{\frac{p}{4}}\right]\\
&\leq 4^{\frac{p-2}{2}}2^{\frac{p}{2}}C_{1}\Ex\left[\frac{1}{2\gamma}\sup_{t\in[0,T]}\norm{X(t)-z}^{p}+\frac{\gamma}{2}\left(\int_{0}^{T}\norm{\sigma(t,x)}_{\rm F}^{2}\dif t\right)^{p/2}\right].
%&=\frac{1}{2}\Ex\left(\sup_{t\in[0,T]}\norm{X(t)-z}^{p}\right)+\alpha_{p}\cdot c(\infty)^{p/2},
\end{align*}
where $\gamma>0$ is an arbitrary positive constant and $C_{1}>0$ is a universal constant. Choosing $\gamma\eqdef 4^{\frac{p-2}{2}}2^{\frac{p}{2}}C_{1}$, we are left with the expression 
$$
4^{\frac{p-2}{2}}2^{\frac{p}{2}}\Ex\left[\sup_{t\in[0,T]}\abs{M_{z}(t)}^{\frac{p}{2}}\right]\leq\frac{1}{2}\Ex\left(\sup_{t\in[0,T]}\norm{X(t)-z}^{p}\right)+C_{2}\Ex[c(T)^{p/2}],
$$
where $C_{2}\eqdef C^{2}_{1}2^{3p-5}.$ Hence, for all $T>0$
\begin{align*}
\Ex\left[\sup_{t\in[0,T]}\norm{X(t)-z}^{p}\right]\leq 2^{p-1}\Ex\left[\norm{X(0)-z}^{p}\right]+2^{\frac{3p}{2}-1}H(T)^{p/2}+C_{3}\Ex[c(T)^{p/2}],
\end{align*}
where $C_{3}\eqdef 2^{\frac{3p}{2}-1}+2C_{2}$. We conclude that $\Ex[\sup_{0\leq t\leq T}\norm{X(t)}^{p}]<\infty$ for all $T>0$.

In the regime \eqref{eq:ASV}, we can strengthen the above bound to 
\begin{align*}
\Ex\left[\sup_{t\in[0,T]}\norm{X(t)-z}^{p}\right]\leq 2^{p-1}\Ex\left[\norm{X(0)-z}^{p}\right]+2^{\frac{3p}{2}-1}H(\infty)^{p/2}+C_{3}\Ex[c(\infty)^{p/2}],
\end{align*}
where $H(\infty)=\lim_{T\to\infty}H(T)$ and $c(\infty)=\lim_{T\to\infty}c(T)$, limits whose existence is justified by dominated convergence. This shows that $\Ex[\sup_{t\geq 0}\norm{X(t)}^{p}]<\infty$. 

For (ii), we start with the estimate 
$$
\frac{1}{2}\norm{X(t)-z}^{2}\leq \frac{1}{2}\norm{X(0)-z}^{2}+\int_{0}^{\infty}\beta(t)h_{C}(p/\beta(t))\dif t+M_{z}(t)+c(t) . 
$$
Setting $A_{t}\eqdef c(t)+\int_{0}^{\infty}\beta(t)h_{C}(p/\beta(t))\dif t$, we immediately obtain from Proposition \ref{prop:RS}, that $\lim_{t\to\infty}\norm{X(t)-z}$ exists and is finite almost surely. Hence, for every $z\in\zer(\opA+\NC_{C})$, there exists a set $\Omega_{z}\subset\Omega$ with $\Pr(\Omega_{z})=1$, such that 
$$
(\forall\omega\in\Omega_{z}):\quad\lim_{t\to\infty}\norm{X(t,\omega)-z}=\tau_{z}(\omega)\in[0,\infty). 
$$
Pick a countable dense subset $\scrZ\subset\zer(\opA+\NC_{C})$ with $\cl(\scrZ)=\zer(\opA+\NC_{C})$ and define $\bar{\Omega}\eqdef\bigcap_{z\in\scrZ}\Omega_{z}$. Then $\Pr(\bar{\Omega})=1-\Pr(\bar{\Omega}^{c})=1-\Pr(\bigcup_{z\in\scrZ}\Omega^{c}_{z})\geq 1-\sum_{z\in\scrZ}\Pr(\Omega_{z}^{c})=1$. Hence, $\Pr(\bar{\Omega})=1$. Fix $z\in\zer(\opA+\NC_{C})$ and consider a sequence $z_{i}\to z$ with $z_{i}\in\scrZ$ for all $i\in\N$. Let
$$
\tau_{i}(\omega)\eqdef\lim_{t\to\infty}\norm{X(t,\omega)-z_{i}}\quad \text{\SQ{for all }} \omega\in\Omega_{z_i}.
$$
It then easily follows for $\omega\in\bar{\Omega}$:
$$
\tau_{i}(\omega)-\norm{z_{i}-z}\leq \lim_{t\to\infty}\norm{X(t,\omega)-z}\leq \tau_{i}(\omega)+\norm{z_{i}-z}.
$$
Hence, $\lim_{i\to\infty}\tau_{i}(\omega)$ exists for all $\omega\in\bar{\Omega}$ and we deduce from this the existence of 
$\lim_{t\to\infty}\norm{X(t)-z}$, except on a set of $\Pr$-measure zero. %Additionally, the convergence implies the a.s. boundedness, $\sup_{t\geq 0}\norm{X(t,\omega)}<\infty$ for all $\omega\in\bar{\Omega}$. 

To prove  (iii), assume \revise{that \eqref{eq:ASV} applies}, and fix a reference point $z\in\zer(\opA+\NC_{C})$. Then, there exists $p\in\NC_{C}(z)$ with $-p\in\opA(z)$. From \eqref{eq:Used}, we have 
\begin{align*}
E_{z}(t)-E_{z}(0)+\int_{0}^{t}\frac{\beta(r)}{2}\Psi(X(r))\dif r & \leq -\int_{0}^{t}\frac{\beta(s)}{2}\Psi(X(s))\dif s+\int_{0}^{t}\inner{p,X(s)-z}\dif s+M_{z}(t)+c(t)\\
&\leq \int_{0}^{t}\frac{\beta(s)}{2}\left[\Psi^{\ast}\left(\frac{2p}{\beta(s)}\right)-\support_{C}\left(\frac{2p}{\beta(s)}\right)\right]\dif s+c(t)+M_{z}(t).
\end{align*}
Rearranging and taking expectations on both sides, leaves us with 
\begin{align*}
\frac{1}{2}\Ex\left(\int_{0}^{t}\beta(r)\Psi(X(s))\dif s\right)&\leq \Ex[ E_{z}(0)]+\Ex[c(t)]+ \int_{0}^{t}\frac{\beta(s)}{2}\left[\Psi^{\ast}\left(\frac{2p}{\beta(s)}\right)-\support_{C}\left(\frac{2p}{\beta(s)}\right)\right]\dif s\\ 
& \leq   \Ex[ E_{z}(0)]+\Ex[c(\infty)]+ \int_{0}^{\infty}\frac{\beta(s)}{2}\left[\Psi^{\ast}\left(\frac{2p}{\beta(s)}\right)-\support_{C}\left(\frac{2p}{\beta(s)}\right)\right]\dif s<\infty
\end{align*}
The random variable $I_{t}(\omega)\eqdef\int_{0}^{t}\beta(r)\Psi(X(s,\omega))\dif s$ is non-negative and non-decreasing. The monotone convergence theorem gives 
\begin{align*}
\Ex[\int_{0}^{\infty}\beta(s)\Psi(X(s))\dif s]&=\lim_{t\to\infty}\Ex[I_{t}]\\
&\leq 2\Ex[E_{z}(0)]+\int_{0}^{\infty}\beta(s) h_{C}\left(\frac{2p}{\beta(s)}\right)\dif s+2\Ex[c(\infty)]<\infty.
\end{align*}
A non-negative random variable with finite expectation is finite almost surely. Hence,
\[
\int_{0}^{\infty}\beta(r)\Psi(X(s))\dif s<\infty\qquad \Pr-\text{a.s.} 
\]
Since $t\mapsto \beta(t)$ is positive and non-decreasing, we have $\beta(t)\geq\beta(0)>0$. Therefore, 
\[
\int_{0}^{\infty}\Psi(X(s))\dif s\leq\frac{1}{\beta(0)}\int_{0}^{\infty}\beta(s)\Psi(X(s))\dif s<\infty\qquad \SQ{\Pr-}\text{a.s.}
\]
If on some sample path $\omega$ it would hold $\liminf_{t\to\infty}\Psi(X(t,\omega))\geq\delta>0$, then for sufficiently large $T$, $\Psi(X(t,\omega))\geq\frac{\delta}{2}$ for all $t\geq T$. This would imply $\int_{T}^{\infty}\Psi(X(t,\omega))\dif t=\infty$, a contradiction. Hence, $\liminf_{t\to\infty}\Psi(X(t))=0$ \SQ{$\Pr$-}a.s. 
\end{proof}
%%%
We next discuss weak ergodic convergence of the stochastic process. Our first result is a rate estimate on the Ces\`{a}ro average of the process $X$ in terms of a restricted merit function associated with problem \eqref{eq:MI}. To establish this result, we define the Br\'{e}zis-Haraux function associated with a maximally monotone operator $\opM:\R^{d}\to 2^{\R^{d}}$ \cite{Brezis:1976aa} as 
\begin{equation}
G_{\opM}(x,u)\eqdef\sup_{(y,v)\in\gr(\opM)}\inner{x-y,v-u}.
\end{equation}
The following properties of $G_{\opM}$ are well known \cite{attouch2018asymptotic,combettes2025lower}:
\begin{align*}
    &G_{\opM}(x,u)= 0\qquad \forall (x,u)\in\gr(\opM),\text{ and }\\
    &G_{\opM}(x,u)\geq \norm{x-(\Id+\opM)^{-1}(x+u)}^{2}\geq 0\quad \text{\SQ{for all }} (x,u)\in\R^{d}\times\R^{d}.
\end{align*}
In order to obtain a quantitative convergence result of the ergodic average trajectory, we focus on the level sets of the function $\gap(x)\eqdef G_{\opA+\NC_{C}}(x,0)$, which can be expressed more explicitly as 
$$
\gap(x)=\sup_{y\in C\cap\dom(\opA)}\sup_{y^{*}\in(\opA+\NC_{C})(y)}\inner{x-y,v}.
$$
The function $x\mapsto \gap(x)\in[0,\infty]$ is convex and lower semi-continuous, as the supremum over affine functions. By Assumption \ref{ass:OpMonotone}, it is proper. To obtain stronger results, we impose the following Slater type constraint qualification on the data of \eqref{eq:MI}. 
\begin{assumption}\label{ass:Slater}
$\Int(\dom(\opA))\cap\Int(C)\neq\emptyset.$
\end{assumption}
Assumption \ref{ass:OpMonotone} guarantees that $\gap$ is proper and Assumption \ref{ass:Slater} yields the identification of $\scrS$ with the zeros of the gap function \cite[][Lemma 4]{burachik1998generalized}: 
$\gap(x)=0 \iff x\in\scrS.$ Hence, $x\mapsto \gap(x)$ is a gap function for problem \eqref{eq:MI}\footnote{A function $g:\scrH\to\scrH$ is called a gap function associated with a monotone operator $\opM$ if $g(x)\geq 0$ for all $x\in\scrH$ and $g(x)=0$ whenever $x\in \opM^{-1}(0)$.}. In fact, \cite{BorDut16} showed that $x\mapsto \gap(x)$ is the minimal translation invariant gap function for \eqref{eq:MI}. To cope with the possible scenario in which $\dom(\opA)\cap C$ is unbounded, we consider a restricted merit function criterion first introduced in \cite{Nes07}. We require a mild continuity property on the operator $\opA:\R^{d}\to 2^{\R^{d}}$. 
\begin{assumption}\label{ass:USC}
    $\opA:\R^{d}\to 2^{\R^{d}}$ is upper semi-continuous (USC) on $C$:  For all $x\in C$ and for all open sets $W\subset\R^{d}$ such that $\opA(x)\subset W$, there exists an open neighborhood $U_{x}\subset\R^{d}$ of $x$ for which $\opA(x')\subset W$ for all $x'\in U_{x}$. 
\end{assumption}
Note that Assumption \ref{ass:USC} implicitly assumes $C\subseteq\dom(\opA)$. By Assumption \ref{ass:OpMonotone}, we can choose a point $a\in\Int(\dom\opA)$ and $\delta>0$ sufficiently small so that 
$B_{\delta}\eqdef \{x\in C\vert\; \norm{x-a}\leq \delta\}\subset \dom(\opA).$ The \emph{restricted gap function} is the mapping $\gap_{\delta}:\R^{d}\to\R$ defined as
\begin{equation}\label{eq:restrictedgap}
\gap_{\delta}(x)\eqdef \sup\{\inner{y^{*},x-y}\vert y\in B_{\delta},y^{*}\in(\opA+\NC_{C})(y)\}.
\end{equation}
It's role as a local optimality measure is made precise in Lemma \ref{lem:gap}, which extends Lemma 1 in \cite{Nes07} to the set-valued case.
\begin{lemma}
\label{lem:gap}
For any $x\in B_{\delta}$, we have $\Theta_{\delta}(\bar{x})\geq 0$. Moreover, if $\bar{x}\in B_{\delta}\cap\scrS$ , then $\Theta_{\delta}(\bar{x})=0$. On the other hand, if $\Theta_{\delta}(x)=0$ for some $x\in C$ with $\norm{x-a}<\delta$, then $x\in\scrS$. 
\end{lemma}
\begin{proof}
See Appendix \ref{app:proofs}.
 \end{proof}
 %\MS{We have to split the theorem below in two parts. Part 1 establishes the upper bound on $\Theta_{\delta}$. Part 2 the same for the clipped merit function $\hat{\Theta}_{\delta}$.}
Using the restricted gap function as our local merit function, we first establish an $O(1/t)$ upper bound along the Cesaro average of the stochastic process $X$. Since the non-negativity of the gap function can only be ensured for points within the neighborhood $B_{\delta}$, we need to establish a lower bound on the gap when evaluated on points which are eventually not inside this localizing set. Such a clipped gap function is introduced before the statement of Proposition \ref{prop:ergodic3}
\begin{proposition}\label{prop:ergodic2}
Define the Ces\`{a}ro average of the process $X$ as  
\begin{equation}\label{eq:cesaro}
\bar{X}(\omega,t)=\Avg(X(\omega,\cdot);0,t)=\frac{1}{t}\int_{0}^{t}X(\omega,s)\dif s.
\end{equation}
Let $a\in\Int(\dom(\opA))\cap\Int(C)$ and $B_{\delta}=\bar{\ball}(a,\delta)\subset\Int(\dom(\opA))\cap\Int(C)$. If Assumption \ref{ass:Noisebound}\eqref{eq:ASV} holds true, then we have $\Ex[\Theta_{\delta}(\bar{X}(t))]\leq O(1/t)$ for all $t>0$. 
\end{proposition}
\begin{proof}
Let $(z,z^{\ast})\in\gr(\opA+\NC_{C})$ so that there exists $p\in\NC_{C}(z)$ with $z^{\ast}-p\in\opA(z)$. Following the steps leading to eq. \eqref{eq:Used}, we obtain
$$
E_{z}(t)-E_{z}(0)+t\inner{z^{\ast},\bar{X}(t)-z}\leq -\int_{0}^{t}\beta(s)\Psi(X(s))\dif s +\int_{0}^{t}\inner{p,X(s)-z}\dif s+M_{z}(t)+c(t).
$$
%where $M_{z}(t)\eqdef \int_{0}^{t}\inner{\sigma^{\ast}(t,X(t))(X(t)-z),\dif W(t)}$ and $c(t)=\int_{0}^{t}\Sigma^{2}(s)\dif s$. 
Rearranging terms gives us 
\begin{equation}\label{eq:ergodic1}
\inner{\bar{X}(t)-z,z^{\ast}}\leq \frac{1}{t}\int_{0}^{t}\beta(s)h_{C}(p/\beta(s))\dif s+\frac{E_{z}(0)+M_{z}(t)+c(t)}{t}.
\end{equation}
The remaining challenge in the proof is to replace the $z$-dependent martingale term $M_{z}(t)$ with a stochastic process which is independent of the anchor point $z$. This will allow us to take the supremum over $z$ and therefore bring the restricted gap function into the picture. To achieve such a uniform bound, we are following a similar argument as used in \cite{MerStaJOTA18}. Define the auxiliary process $Z=\{(Z(t),\scrF_{t});t\geq 0\}$ by 
\begin{equation}\label{eq:thatsZ}
Z(0)\eqdef X(0), \quad Z(t)\eqdef X(0)-\int_{0}^{t}\sigma(s,X(s))\dif W(s)\quad \text{\SQ{for all }} t\geq 0.
\end{equation}
Let $\pi(t)\eqdef\frac{1}{2}\norm{Z(t)-z}^{2}$, so that by It\^{o}'s formula 
$$
\pi(t)-\pi(0)=\int_{0}^{t}\inner{z-Z(s),\sigma(s,X(s))\dif W(s)}+\frac{1}{2}\int_{0}^{t}\norm{\sigma(s,X(s))}^{2}_{\rm F}\dif s.
$$
Rearranging this, and using the variance bound, we are left with 
\begin{equation}\label{eq:auxbound}
\int_{0}^{t}\inner{Z(s)-z,\sigma(s,X(s))\dif W(s)}\leq \frac{1}{2}\norm{X(0)-z}^{2}+c(t_{1})=E_{z}(0)+c(t). 
\end{equation}
Using this bound, we see for $t\geq 0$:
\begin{align*}
M_{z}(t)&=\int_{0}^{t}\inner{X(s)-z,\sigma(s,X(s))\dif W(s)}\\
&=\int_{0}^{t}\inner{X(s)-Z(s),\sigma(s,X(s))\dif W(s)}+\int_{0}^{t}\inner{Z(s)-z,\sigma(s,X(s))\dif W(s)}\\
&\leq \int_{0}^{t}\inner{X(s)-Z(s),\sigma(s,X(s))\dif W(s)}+E_{z}(0)+c(t)\\
&=U(t)+E_{z}(0)+c(t).
\end{align*}
The process $U\eqdef \{(U(t),\scrF_{t});t\geq 0\}$ defined by 
\begin{equation}\label{eq:thatsU}
U(t)\eqdef \int_{0}^{t}\inner{X(s)-Z(s),\sigma(s,X(s))\dif W(s)}\qquad t\geq 0,
\end{equation}
is a continuous martingale. Substituting this into \eqref{eq:ergodic1}, we can continue as follows: 
\begin{equation}\label{eq:MainInner}
\inner{\bar{X}(t)-z,z^{\ast}}\leq \frac{1}{t}\int_{0}^{t}\beta(s)h_{C}(p/\beta(s))\dif s+\frac{2}{t}[E_{z}(0)+c(t)]+\frac{1}{t}U(t).
\end{equation}
Our construction thus produced an upper bound which is independent of the pair $(z,z^{\ast})\in\gr(\opA+\NC_{C})$, besides the choice of the normal cone element $p\in\NC_{C}(z)$. However since $z\in B_{\delta}$, we have $\NC_{C}(z)=\{0\}$, so that the above display reduces to 
$$
\inner{\bar{X}(t)-z,z^{\ast}}\leq \frac{2}{t}[E_{z}(0)+c(t)]+\frac{1}{t}U(t).
$$
We now make use of the definition of the restricted merit function so that 
\begin{equation}\label{eq:gapintermeddiate}
\sup_{z\in B_{\delta},z^{*}\in(\opA+\NC_{C})(z)}\inner{\bar{X}(t)-z,z^{\ast}}\leq \frac{2}{t}[\sup_{z\in B_{\delta}}E_{z}(0)+c(t)]+\frac{1}{t}U(t)= \frac{D_{\delta}(t)+U(t)}{t}.
\end{equation}
Because $U$ is a martingale, we have $\Ex[U(t)]=0$, and it follows 
$$
\Ex\left[\Theta_{\delta}(\bar{X}(t))\right]\leq \frac{D_{\delta}(t)}{t}.%\frac{2}{t}[\sup_{z\in B_{\delta}}E_{z}(0)+c(t)].
$$
Finally, via \eqref{eq:ASV}, we deduce 
$$
\Ex[\Theta_{\delta}(\bar{X}(t)]\leq \frac{\ca_{\delta}}{t},\text{ where }\ca_{\delta}\eqdef 2\Ex\left[\sup_{z\in B_{\delta}}E_{z}(0)\right]+\int_{0}^{\infty}\Sigma^{2}(s)\dif s. 
$$
\end{proof}
\revise{To obtain a measure which is lower bounded, consider the clipped gap function 
$$
\hat{\Theta}_{\delta}(x)=\max\{0,\Theta_{\delta}(x)\}\equiv[\Theta_{\delta}(x)]_{+}.
$$
Since $\hat{\Theta}_{\delta}(x)\geq\Theta_{\delta}(x)$, if follows from Lemma \ref{lem:gap} for $x\in B_{\delta}$ that
$\hat{\Theta}_{\delta}(x)\Rightarrow\Theta_{\delta}(x)=0\Rightarrow x\in \scrS.$ Hence, we obtain a gap function which is globally non-negative. 
\begin{proposition}
    \label{prop:ergodic3}
    Under the same conditions as in Proposition \ref{prop:ergodic2}, we have 
    $$
    0\leq \Ex[\hat{\Theta}_{\delta}(\bar{X}(t))]\leq O(1/t) \qquad\forall t>0.
    $$
\end{proposition}
\begin{proof}
Continuing from eq. \eqref{eq:gapintermeddiate}, we obtain the almost sure relation 
$$
\hat{\Theta}_{\delta}(\bar{X}(t))\leq \frac{D_{\delta}(t)+[U(t)]_{+}}{t}.
$$
Taking expectations on both sides, we arrive at 
$$
\Ex[\hat{\Theta}_{\delta}(\bar{X}(t))]\leq \frac{D_{\delta}(t)+\Ex([U(t)]_{+})}{t}.
$$
It remains to bound $\Ex([U(t)]_{+})$. We start from \eqref{eq:thatsZ} and develop a bound for the process $(Z(t))_{t\geq 0}$. From Doob's inequality and It\^{o} isometry we obtain 
\begin{align*}
\Ex\left[\sup_{0\leq t\leq T}\norm{\int_{0}^{t}\sigma(s,X(s))\dif W(s)}^{2}\right]&\leq 4\Ex\left[\norm{\int_{0}^{T}\sigma(s,X(s))\dif W(s)}^{2}\right]\\
&=4\Ex\left[\int_{0}^{T}\norm{\sigma(s,X(s))}^{2}_{F}\dif s\right]\\
&\leq 4\int_{0}^{\infty}\Sigma(s)^{2}\dif s.
\end{align*}
Since $\sup_{0\leq t\leq T}\norm{Z(t)}^{2}\leq 2\norm{X(0)}^{2}+2\sup_{0\leq t\leq T}\norm{\int_{0}^{t}\sigma(s,X(s))\dif W(s)}^{2}$, we arrive at
\begin{equation}\label{eq:boundZ}
\Ex\left[\sup_{0\leq t\leq T}\norm{Z(t)}^{2}\right]\leq 2\norm{X(0)}^{2}+8\int_{0}^{\infty}\Sigma(s)^{2}\dif s<\infty.
\end{equation}
From the definition of $U(t)$ in \eqref{eq:thatsU}, we can bound its quadratic variation process as
\[
\Ex([U]_{t})\leq\int_{0}^{t}\Sigma(s)^{2}\Ex(\norm{X(s)-Z(s)}^{2})\dif s\leq \Ex(\sup_{r\geq 0}\norm{X(r)-Z(r)}^{2})\cdot\int_{0}^{\infty}\Sigma(s)^{2}\dif s.
\]
Moreover, $\sup_{r\geq0}\norm{X(r)-Z(r)}^{2}\leq 2\sup_{r\geq 0}\norm{X(r)}^{2}+2\sup_{r\geq 0}\norm{Z(r)}^{2}.$ Proposition \ref{prop:1} controls the first term, while \eqref{eq:boundZ} controls the second. Hence, \eqref{eq:ASV} yields $\sup_{t\geq 0}\Ex([U]_{t})<\infty.$ It\^{o}'s isometry then gives 
\[
\sup_{t\geq 0}\Ex(\abs{U(t)}^{2})=\sup_{t\geq 0}\Ex([U]_{t})<\infty.
\]
In particular, \(U\) is a true, uniformly \(L^2\)-bounded martingale, rather than merely a local martingale, and
\[
\sup_{t\ge0}\Ex([U(t)]_+)
\le
\sup_{t\ge0}\big(\mathbb E|U(t)|^2\big)^{1/2}
<\infty.
\]
We conclude that
$$
\Ex\left[\hat{\Theta}_{\delta}(\bar{X}(t))\right]\leq \hat{\ca}_{\delta}/t=O(1/t),\text{ with }\hat{\ca}_{\delta}\eqdef \ca_{\delta}+\sup_{t\ge0}\Ex([U(t)]_+).
$$
\end{proof}
}
\begin{corollary}
Define the deterministic function $\chi^{x}_{t}\eqdef \Ex[\bar{X}(t)\vert X(0)=x]$. Under Assumption \ref{ass:Noisebound}\eqref{eq:ASV}, every accumulation point of $\chi^{x}_{\bullet}$ belongs to $\scrS=\zer(\opA+\NC_{C})$.
\end{corollary}
\begin{proof}
Starting from eq. \eqref{eq:MainInner}, for arbitrary $(z,z^{*})\in\gr(\opA+\NC_{C})$, we immediately see
\begin{align*}
\Ex[\inner{\bar{X}(t)-z,z^{\ast}}\vert X(0)=x]=\inner{\Ex[\bar{X}(t)\vert X(0)=x]-z,z^{\ast}}\leq \frac{1}{t}\int_{0}^{t}\beta(s)h_{C}(p/\beta(s))\dif s+\frac{2}{t}[E_{z}(0)+c(t)].
\end{align*}
We conclude that the deterministic process $\chi^{x}_{t}=\Ex[\bar{X}(t)\vert X(0)=x]$ satisfies the limit relation 
$$
\lim_{n\to\infty}\inner{\chi^{x}_{t_{n}}-z,z^{\ast}}=\inner{\chi^{x}_{\infty}-z,z^{\ast}}\leq 0. 
$$
This inequality being true for any $(z,z^{\ast})\in\gr(\opA+\NC_{C})$. By maximal monotonicity of the operator $\opA+\NC_{C}$, we obtain $0\in(\opA+\NC_{C})(\chi^{x}_{\infty})$, that is $\chi^{x}_{\infty}\in\zer(\opA+\NC_{C})$ (cf. \eqref{eq:eqi}). 
\end{proof}

If $C$ is bounded, there is no need for the truncation procedure and we can consider the unrestricted gap function $\gap(x)$. In this case, the bound in Proposition \ref{prop:ergodic2} becomes 
$$
\Ex[\gap(\bar{X}(t)]=\frac{1}{t}\int_{0}^{\infty}\beta(s)h_{C}(p/\beta(s))\dif s+\frac{2}{t}[\sup_{z\in C}E_{z}(0)+c(t)]=O(1/t).
$$
%If $C$ is bounded, there is no need for the truncation procedure and we can consider the unrestricted gap function $\gap(x)$. In this case, the bound in Theorem \ref{th:ergodic2} becomes 
%$$
%\Ex[\gap(\bar{X}(t)]=\frac{1}{t}\int_{0}^{\infty}\beta(s)h_{C}(p/\beta(s))\dif s+\frac{2}{t}[\sup_{z\in C}E_{z}(0)+c(t)]=\frac{D}{t}.
%$$
%Under this boundedness assumption, we next derive almost sure convergence properties on the stochastic process directly, assuming that the solution set $\scrS$ satisfies a weak sharpness property \cite{marcotte_weak_1998,Al-Homidan:2017aa}. 

%\begin{assumption}[Weak Sharpness]
%The solution set $\scrS=\zer(\opA+\NC_{C})$ is weak sharp with parameter $\kappa>0$ if 
%    \Theta(x)\geq\kappa\dist(x,\scrS)\qquad\forall x\in C.
%\end{equation}
%\end{assumption}
%Under this geometric setting, Theorem \ref{th:ergodic2} readily gives us a rate on the ergodic trajectory. 
%\begin{corollary}
%Assume $C\subset\R^{d}$ is a compact convex subset and $\scrS$ is weak sharp with parameter $\kappa>0$. Then, 
%\Ex[\dist(\bar{X}(t),\scrS)]=O(1/t).
%$$
%\end{corollary}

\section{Subgradient evolution equations}
\label{sec:subgradient}
%----------------------------------------------------------------------
%%%SubgradientEvolution 
%----------------------------------------------------------------------
%!TEX root = ./MultiscaleStochastic.tex
%

We now add to our analysis the fact that the maximally monotone operator $\opA$ is a subdifferential of a closed convex and proper function $\Phi:\R^{d}\to(-\infty,\infty].$ In this special setting, our definition of a solution to the stochastic differential inclusion \eqref{eq:SDI} particularizes to the following:

%\begin{assumption}\label{ass:SGmain}
%Consider problem \eqref{eq:Opt}, with the following assumptions: 
%\begin{itemize}
%\item $\Phi:\R^{d}\to(-\infty,\infty]$ is a lower semicontinuous and convex function, with $\Int(\dom\Phi)\neq\emptyset$.
%\item $\Psi:\R^{d}\to\R$ is a convex and continuously differentiable function; Its gradient mapping $\nabla\Psi$ is Lipschitz on bounded sets of $\R^{d}$. 
%\item $\setC=\Psi^{-1}(0)=\argmin\Psi\neq\emptyset$. 
%\item $\dom(\Phi)\cap\setC\neq\emptyset$.
%%\item $\beta:\R_{\geq 0}\to(0,\infty)$ a non-decreasing function of class $\bC^{1}$ with $\lim_{t\to\infty}\beta(t)=\infty$. 
%\end{itemize}
%\end{assumption}

\begin{definition}\label{def:SDISubgradient}
Let $\xi\in L^{2}(\Omega;\cl(\dom(\Phi)))$ be a given random variable. A process $X=\{(X(t),\scrF_{t}),t\in [0,T]\}$ is a \revise{strong solution to \eqref{eq:SDI}} if there exists a process $K=\{(K(t),\scrF_{t}),t\in [0,T]\}$ such that conditions (SE.1)-(SE.3) hold, while (SE.4) is replaced by the following:
\begin{enumerate}
%\item[(SE.1)] $(X(t))_{t\in[0,T]}$ is a $\R^{d}$-valued process, $\scrF_{t}$-adapted with continuous sample paths;
%\item[(SE.2)] $(K(t))_{t\in[0,T]}$ is a $\R^{d}$-valued process, $\scrF_{t}$-adapted whose sample paths are continuous and of bounded variation on $[0,T]$. 
%\item[(SE.3)] $\dif X(t)=-\beta(t)\nabla\Psi(X(t))\dif t+\sigma(t,X(t))\dif W(t)-\dif K(t)$ holds a.e. $t\in[0,T]$ and $\Pr$-a.s. 
%\item[(SE.4)] $X(0)=\xi$ $\Pr$-a.s.
\item[(SE.4')] For all $z\in\R^{d}$, and $0\leq r\leq s\leq T$, we have 
\begin{equation}\label{eq:Ksubgradient}
\int_{r}^{s}\inner{z-X(u),\dif K(u)}+\int_{r}^{s}\Phi(X(u))\dif u\leq (s-r)\Phi(z).
\end{equation}
\end{enumerate}
\end{definition}
\begin{remark}
Note that condition (SE.4') is equivalent to 
\begin{equation}\label{eq:KSubgradient}
\int_{r}^{s}\inner{X(u)-z,\dif K(u)-\hat{z}\dif u}\geq 0 
\quad\Pr-\text{a.s.}\quad
\forall(z,\hat{z})\in\gr(\partial\Phi),\forall 0\leq r\leq s\leq T.
\end{equation}
In this sense, this condition can be written formally as 
$$
\dif K(t)\in\partial\Phi(X(t))(\dif t)\qquad (\omega,t)\in\Omega\times[0,T].
$$
\end{remark}

\begin{lemma}\label{lem:bound1}
Let Assumptions \ref{ass:OpMonotone}-\ref{ass:Noisebound}\eqref{eq:ASV} be in place. Then, for every $z\in\scrS$ and $p\in\NC_{C}(z)$ \revise{so that $-p\in\partial\Phi(z)$}, we have 
\begin{itemize}
\item[(i)] $\Ex\left[\int_{0}^{t}\left(\beta(s)\Psi(X(s))+\Phi(X(s))-\Phi(z)\right)\dif s\right]$ converges in $\R$.
\item[(ii)] $\Ex\left[\int_{0}^{t}\inner{-p,X(s)-z}\dif s\right]$ converges in $\R$.
\end{itemize}
\end{lemma}
\revise{
\begin{proof}
Consider the function $E_{z}(t)=\frac{1}{2}\norm{X(t)-z}^{2}$ for $z\in\scrS$. It\^{o}'s formula yields 
\begin{equation}\label{eq:subgrad1}
E_{z}(t) +\int_{0}^{t}\beta(r)\Psi(X(r))\dif r+\int_{0}^{t}[\Phi(X(r))-\Phi(z)]\dif r\leq E_{z}(0)+M_{z}(t)+c(t),
\end{equation}
where $M_{z}(t)\eqdef\int_{0}^{t}\inner{X(r)-z,\sigma(r,X(r))\dif W(r)}$ and $c(t)\eqdef \frac{1}{2}\int_{0}^{t}\norm{\sigma(r,X(r))}^{2}_{\rm F}\dif r$. Taking expectations on both sides, and calling $\Gamma(t)\eqdef\Ex[E_{z}(0)+c(t)]$, we obtain the upper bound 
\begin{equation}\label{eq:upperbound}
\Ex\left[E_{z}(t)+\int_{0}^{t}\beta(r)\Psi(X(r))\dif r+\int_{0}^{t}[\Phi(X(r))-\Phi(z)]\dif r\right]\leq\Gamma(t)\quad \text{\SQ{for all }} t\geq 0.
\end{equation}
Using Assumption \ref{ass:Noisebound}\eqref{eq:ASV}, it readily follows 
$$
\Gamma(t)\leq \Gamma_{\infty}\eqdef \Ex[ E_{z}(0)+\int_{0}^{\infty}\Sigma^{2}(r)\dif r]<\infty\quad \text{\SQ{for all }} t\geq 0.
$$
Now pick $z\in\scrS$, so that there exists $p\in\NC_{C}(z)$ with $-p\in\partial\Phi(z)$. Hence, the convex sub-gradient inequality yields
\begin{equation}\label{eq:subgradientPhi}
\Phi(X(t))-\Phi(z)\geq \inner{-p,X(t)-z}\qquad \SQ{\text{$\Pr-$a.s. }, \text{for all }} t\geq 0.
\end{equation}
This yields the relation 
\begin{align*}
E_{z}(t) +\int_{0}^{t}\beta(r)\Psi(X(r))\dif r+&\int_{0}^{t}\inner{-p,X(r)-z}\dif r\\
&\leq E_{z}(t) +\int_{0}^{t}\beta(r)\Psi(X(r))\dif r+\int_{0}^{t}[\Phi(X(r))-\Phi(z)]\dif r.
\end{align*}
Define 
\begin{align}
    g(t)&\eqdef \beta(t)\Psi(X(t))+\Phi(X(t))-\Phi(z),\\
    j(t)&\eqdef \beta(t)\Psi(X(t))+\inner{-p,X(t)-z},
\end{align}
so that $g(t)\geq j(t).$ By Fenchel-Young, we obtain 
\[
j(t)\geq -\beta(t)\left[\Psi^{*}\left(\frac{p}{\beta(t)}\right)-\support_{C}\left(\frac{p}{\beta(t)}\right)\right]=-\beta(t)h_{C}\left(\frac{p}{\beta(t)}\right).
\]
Setting 
\begin{equation}\label{eq:thatsa}
a(t)\eqdef \beta(t)h_{C}\left(\frac{p}{\beta(t)}\right),
\end{equation}
we obtain the relation
$$
g(t)\geq j(t)\geq-a(t).
$$
The Attouch-Czarnecki condition gives $\int_{0}^{\infty}a(t)\dif t<\infty.$ Furthermore, eq. \eqref{eq:upperbound} yields 
$$
\Ex[\int_{0}^{t}g(r)\dif r]\leq \Gamma(t)\leq\Gamma_{\infty}. 
$$
Partition the function $g$ into its positive and negative part $g(t)=g^{+}(t)-g^{-}(t)$. If $g^{+}(t)=0$, then $g^{-}(t)\geq 0$, and thus $-a(t)\leq g(t)=-g^{-}(t)$. If $g^{+}(t)>0$, then $g^{-}(t)=0$, and $-a(t)\leq 0=-g^{-}(t)$. Together, we conclude $g^{-}(t)\leq a(t)$ for all $t\geq 0$. Since both functions are non-negative, we conclude further 
$$
\Ex\left[\int_{0}^{t}g^{-}(r)\dif r\right]\leq\int_{0}^{t}a(r)\dif r\quad \text{\SQ{for all }} t\geq 0. 
$$
Hence, the positive and increasing function $t\mapsto \int_{0}^{t}\Ex[g^{-}(r)]\dif r$ has a limit. Since as well 
\[
\Ex\left[\int_{0}^{t}g^{+}(r)\dif r\right]=\Ex\left[\int_{0}^{t}g(r)\dif r\right]+\Ex\left[\int_{0}^{t}g^{-}(r)\dif r\right]\leq \Gamma_{\infty}+\int_{0}^{\infty}a(t)\dif t
\]
the positive and increasing function $t\mapsto\Ex\left[\int_{0}^{t}g^{+}(r)\dif r\right]$ also has a limit. Hence, 
$$
t\mapsto \Ex\left[\int_{0}^{t}g(r)\dif r\right]=\Ex\left[\int_{0}^{t}g^{+}(r)\dif r\right]-\Ex\left[\int_{0}^{t}g^{-}(r)\dif r\right]
$$
has a limit as the difference of two converging integrals. 

For statement (ii), first recall the already established relation $j(t)\geq-a(t)$, which implies $j^{-}(t)\leq a(t)$. Hence, by Attouch-Czarnecki, 
$$
\Ex\left[\int_{0}^{\infty}j^{-}(r)\dif r\right]\leq \int_{0}^{\infty}a(r)\dif r<\infty. 
$$
To bound the positive part, we essentially repeat the previous argument. Indeed, 
$$
\Ex\left[\int_{0}^{t}j^{+}(r)\dif r\right]=\Ex\left[\int_{0}^{t}j(r)\dif r\right]+\Ex\left[\int_{0}^{t}j^{-}(r)\dif r\right]\leq \Gamma_{\infty}+\int_{0}^{\infty}a(r)\dif r<\infty. 
$$
Hence, 
$$
t\mapsto\int_{0}^{t}\Ex[j(r)]\dif r=
\Ex\left[\int_{0}^{t}j^{+}(r)\dif r\right]-\Ex\left[\int_{0}^{t}j^{-}(r)\dif r\right]
$$
converges to a finite limit. Now write 
\begin{align*}
    \Ex\left[\int_{0}^{t}\inner{-p,X(r)-z}\dif r\right]&=\Ex\left[\int_{0}^{t}j(r)\dif r\right]-\Ex\left[\int_{0}^{t}\beta(r)\Psi(X(r))\dif r\right]
\end{align*}
By Proposition \ref{prop:1}(iii), we conclude that $\lim_{t\to\infty}\Ex\left[\int_{0}^{t}\inner{-p,X(r)-z}\dif r\right]$ exists in $\R$, which demonstrates (ii).
\end{proof}
}

\begin{lemma}\label{lem:ASlimits}
Let Assumptions \ref{ass:OpMonotone}-\ref{ass:Noisebound}\eqref{eq:ASV} be in place. For every $z\in\scrS=\argmin_{x\in C}\Phi(x)=\zer(\partial\Phi+\NC_{C})$ and $p\in\NC_{C}(z)$ \revise{with $-p\in\partial\Phi(z)$}, it holds true that
\begin{itemize}
\item[(i)] $\liminf_{t\to\infty}\inner{-p,X(t)-z}\leq 0\quad\Pr$-a.s.
\item[(ii)] $\liminf_{t\to\infty}\abs{\Phi(X(t))-\Phi(z)+\beta(t)\Psi(X(t))}= 0\quad\Pr$-a.s. 
\end{itemize}
\end{lemma}
\revise{
\begin{proof}
(i) To simplify the notation, denote by $q(t,\omega)\eqdef\inner{-p,X(t,\omega)-z}.$ Invoking the definitions introduced in the proof of Lemma \ref{lem:bound1}, we then have $j(t)=\beta(t)\Psi(X(t))+q(t)$. For the sake of obtaining a contradiction, suppose that 
\[
\Pr\left(\liminf_{t\to\infty}q(t)>0\right)>0.
\]
Since, 
\[
\left\{\liminf_{t\to\infty}q(t)>0\right\}=\bigcup_{m,k\geq 1}\left\{q(t)\geq\frac{1}{m}\quad\forall t\geq k\right\},
\]
there exists $\delta>0,T_{0}>0$ and a measurable set $\Omega_{0}\subset\Omega$ such that 
$$
\Pr(\Omega_{0})>0\text{ and }q(t,\omega)\geq\delta\quad\text{\SQ{for all} }t\geq T_{0}\text{ \SQ{and all} }\omega\in\Omega_{0}.
$$
We claim that
\begin{equation} \label{E:q-}
    q^{-}(t)\leq a(t)+\beta(t)\Psi(X(t))\quad\text{ \SQ{for all} } t\geq 0,
\end{equation}
where $a:\R_{\geq 0}\to\R_{\geq 0}$ is defined in \eqref{eq:thatsa}. If $q^{-}(t)=0$, the claim is clearly true. If $q^{-}(t)>0$, then $q^{+}(t)=0$, and thus 
$$
-q^{-}(t)=q(t)=j(t)-\beta(t)\Psi(X(t))\geq -a(t)-\beta(t)\Psi(X(t)).
$$
Departing from \eqref{E:q-}, we can conclude 
\[
\Ex\left[\int_{0}^{t}q^{-}(r)\dif r\right]\leq \int_{0}^{t}a(r)\dif r+\Ex\left[\int_{0}^{t}\beta(r)\Psi(X(r))\dif r\right].
\]
In conjunction with Proposition \ref{prop:1}(iii), we thus obtain 
\[
\Ex[\int_{0}^{\infty}q^{-}(t)\dif t]<\infty.
\]
We continue by observing that for all $T>T_{0}$
\begin{align*}
    \Ex\left[\int_{0}^{T}q(t)\dif t\right]&=\Ex\left[\int_{0}^{T}q^{+}(t)\dif t\right]-\Ex\left[\int_{0}^{T}q^{-}(t)\dif t\right]\\
    &\geq \int_{T_{0}}^{T}\Ex\left[\mathds{1}_{\Omega_{0}}q(t)\dif t\right]-\Ex\left[\int_{0}^{T}q^{-}(t)\dif t\right]\\
    &\geq \delta\Pr(\Omega_{0})(T-T_{0})-\Ex\left[\int_{0}^{T}q^{-}(t)\dif t\right].
\end{align*}
As $T\uparrow\infty$ we would then have $\Ex[\int_{0}^{\infty}q(t)\dif t]=\infty$, a contradiction to Lemma \ref{lem:bound1}(i).

(ii) Reconsider the function $t\mapsto g(t)=\Phi(X(t))-\Phi(z)+\beta(t)\Psi(X(t))$, where $z\in\scrS$. Since $\abs{g}=g^{+}+g^{-}$, it follows from Lemma \ref{lem:bound1}(i) that $\int_{0}^{\infty}\abs{g(t)}\dif t<\infty$ \SQ{$\Pr$-}a.s. 

\end{proof}
}

The main theorem of this section is the following asymptotic convergence result in terms of feasibility and optimality of the stochastic process $X$. 
\begin{theorem}\label{th:2}
Let Assumptions \ref{ass:OpMonotone}-\ref{ass:Noisebound}\eqref{eq:ASV} be in place. Then, the following statements hold true:
\begin{itemize} 
\item[(a)] %All limit points of the trajectory $X$ are $\Pr$-a.s. contained in $\scrS=\argmin_{C}(\Phi)$.
The process $X$ converges almost surely to an $\scrS=\argmin_{C}\Phi$-valued random variable $x_{\infty}.$
\item[(b)] $\lim_{t\to\infty}\Psi(X(t))=0$ \revise{$\Pr$-\text{a.s.}} and if $\Phi$ is continuous, then $\Phi(X(t))\to\min_{C}\Phi$ \revise{$\Pr$-}a.s. 
\end{itemize}
\end{theorem}
\revise{
\begin{proof}
(a) Fix $z\in\scrS$ arbitrary. Lemma \ref{lem:ASlimits}(ii) gives a full measure set $\Omega_{1}\subseteq\Omega$ on which $t\mapsto \int_{0}^{t}g(s,\omega)\dif s$ converges. By Proposition \ref{prop:1}(iii), there exists a full measure set $\Omega_{2}\subseteq\Omega$ on which $t\mapsto \int_{0}^{t}\beta(s)\Psi(X(s,\omega)\dif s$ converges. Hence, on the full measure set $\bar{\Omega}\eqdef\Omega_{1}\cap\Omega_{2}$, we can choose a sequence $(t_{n}(\omega))_{n\in\N}$ with $\lim_{n\to\infty}t_{n}(\omega)=\infty$ and 
\[
\abs{g(t_{n}(\omega),\omega)}+\beta(t_{n}(\omega))\Psi(X(t_{n}(\omega),\omega))\to 0.
\]
Pathwise boundedness implies that, for each $\omega\in\bar{\Omega}$, the sequence $x_{n}(\omega)\eqdef X(t_{n}(\omega),\omega)$ is bounded. For each such $\omega$, we can thus extract a converging subsequence $x_{n_{j}}(\omega)$ with limit $x_{\infty}(\omega)$. Along this converging subsequence, we see 
$$
\Phi(x_{n_{j}}(\omega))-\Phi(z)=g(t_{n_{j}}(\omega))-\beta(t_{n_{j}}(\omega))\Psi(x_{n_{j}}(\omega))\to 0
$$
Since $\Psi$ is continuous, we have the feasibility $\Psi(x_{\infty}(\omega))=\lim_{j\to\infty}\Psi(x_{n_{j}}(\omega))=0$ for all $\omega\in\bar{\Omega}$. Consequently, $x_{\infty}\in\scrC$ almost surely. Additionally, by lower semi-continuity of $\Phi$, 
\[
\Phi(x_{\infty}(\omega))\leq\lim\inf_{j\to\infty}\Phi(x_{n_{j}}(\omega))=\lim_{j\to\infty}\Phi(x_{n_{j}}(\omega))=\Phi(z)=\min_{\scrC}\Phi.
\]
Together with the feasibilty result established above, we conclude $x_{\infty}(\omega)\in\scrS$ for all $\omega\in\bar{\Omega}$. This shows that there exists an accumulation point of the process $(X(t,\omega))_{t\geq 0}$ that belongs to $\scrS$. By Proposition \ref{prop:1}(ii), for every $z\in\scrS$ the limit $\lim_{t\to\infty}\norm{X(t)-z}$ exists almost surely. Hence, in particular $\norm{X(t)-x_{\infty}}$ exists almost surely. Since $\norm{X(t_{n_{j}})-x_{\infty}}\to 0$, it follows $\norm{X(t)-x_{\infty}}\to 0$ as $t\to\infty$ almost surely. Therefore, $X(t)$ converges to an $\scrS$-valued random variable $x_{\infty}(\omega)$, almost surely. 
%%%%%
\\
(b) This is an immediate consequence of (a).
\end{proof}
}
 We complete the asymptotic analysis of the stochastic subgradient flow, by investigating the long-run properties of C\'{e}saro averages of the trajectories. 
 
\begin{proposition}\label{prop:ergodic1}
Let Assumptions \ref{ass:OpMonotone}-\ref{ass:Noisebound}\eqref{eq:ASV} be in place. Consider the C\'{e}saro averaged sample path process $\bar{X}$ as in eq. \eqref{eq:cesaro}. Then, we have 
\begin{equation}\label{eq:ExErgodic}
\Ex[\Phi(\bar{X}(t))-\min_{C}\Phi\vert X(0)=x]\leq \frac{\dist(x,\scrS)^{2}}{2t}+\frac{1}{2t}\int_{0}^{t}\Sigma^{2}(s)\dif s=O(1/t).
\end{equation}
\end{proposition}
\begin{proof}
Let $z\in C\cap\dom(\partial\Phi)$ and $y\in(\partial\Phi+\NC_{C})(z)$, so that there exists $p\in\NC_{C}(z)$ such that $y-p\in\partial\Phi(z)$. Using \eqref{eq:KSubgradient}

\begin{align*}
\dif E_{z}(t)&=\inner{X(t)-z,-\beta(t)\nabla\Psi(X(t))\dif t}-\inner{X(t)-z,\dif K(t)}+\inner{\sigma^{\ast}(t,X(t))(X(t)-z),\dif W(t)}\\
&+\frac{1}{2}\tr[\sigma\sigma^{\ast}(t,X(t))]\dif t\\
&\leq -\beta(t)\inner{X(t)-z,\nabla\Psi(X(t))}\dif t+\inner{X(t)-z,\sigma(t,X(t))\dif W(t)}+\frac{1}{2}\tr[\sigma\sigma^{\ast}(t,X(t))]\dif t\\
&+(\Phi(z)-\Phi(X(t)))\dif t.
\end{align*}
Integration from $0$ to $t$ and rearranging yields 
\begin{align*}
\frac{1}{t}\int_{0}^{t}(\Phi(X(s))-\Phi(z))\dif s&\leq \frac{E_{z}(0)}{t}-\frac{1}{t}\int_{0}^{t}\beta(s)\inner{X(s)-z,\nabla\Psi(X(s))}\dif s\\
&+\frac{1}{t}\int_{0}^{t}\inner{X(s)-z,\sigma(s,X(s))\dif W(s)}+\frac{1}{2t}\int_{0}^{t}\Sigma^{2}(s)\dif s\\
&\leq \frac{E_{z}(0)}{t}-\frac{1}{t}\int_{0}^{t}\beta(s)\Psi(X(s))\dif s \\
&+\frac{1}{t}\int_{0}^{t}\inner{X(s)-z,\sigma(s,X(s))\dif W(s)}+\frac{1}{2t}\int_{0}^{t}\Sigma^{2}(s)\dif s\\
&\leq \frac{E_{z}(0)}{t}+\frac{1}{t}\int_{0}^{t}\inner{X(s)-z,\sigma(s,X(s))\dif W(s)}+\frac{1}{2t}\int_{0}^{t}\Sigma^{2}(s)\dif s.
\end{align*}
\end{proof}

\begin{remark}
 If instead the noise condition \eqref{eq:UBV} holds, it follows from Jensen's inequality 
\begin{align*}
\Ex[\Phi(\bar{X}(t))-\min_{C}\Phi\vert X(0)=x]& \leq \frac{\norm{x-z}^{2}+\int_{0}^{t}\Sigma^{2}(s)\dif s}{2t}\leq \frac{\dist(x,\scrS)^{2}+t\sigma_{*}^{2}}{2t}.
\end{align*}
Hence, we obtain convergence, up to a plateau, as expected.
\end{remark}

\section{Discrete-time analysis and numerical implementation}
\label{sec:discrete} 
%----------------------------------------------------------------------
%%%Numerics	
%----------------------------------------------------------------------
%!TEX root = ./MultiscaleStochastic.tex
%

In this section we derive from the continuous-time dynamical system \eqref{eq:SDIMMO} a discrete-time numerical algorithm for solving \eqref{eq:MI}, and derive finite-time complexity statements. Let $T>0$ be fixed. Let $\pi_{\delta}=\{0=t_{0}<t_{1}<\ldots<t_{N_{\delta}}=T\}$ be a partition of $[0,T]$ with mesh size $\Delta t_{k}^{\delta}\eqdef t_{k+1}-t_{k}$ and $\delta\eqdef\max\{\Delta t_{k}^{\delta}:0\leq k\leq N_{\delta}-1\}.$ We consider the sequence $(X^{\delta}_{n})_{n=0}^{N_{\delta}}$ recursively defined by
\begin{equation}\label{eq:SFB}
\left\{\begin{aligned}
&X^{\delta}_{n+1}=(\Id+\Delta t^{\delta}_{n}\opA)^{-1}(X^{\delta}_{n}-\Delta t^{\delta}_{n}\beta(t_n)(\nabla\Psi(X^{\delta}_{n})+\sigma(t_{n}^{\delta},X^{\delta}_{n})(W_{t_{n+1}}-W_{t_{n}}))),\\
&X_{0}^{\delta}=\xi
\end{aligned}\right.
\end{equation}
The resulting scheme is a noisy version of the splitting method investigated in \cite{AttCzarPey11}. It arises also naturally in the discretization framework of \cite{Bernardin:2003aa}. More precisely, the following consistency theorem is a central result in \cite[][Theorem 7.1]{Bernardin:2003aa}: 
\begin{theorem}\label{th:BerConvergence}
    Consider the discrete-time Markov chain $(X^{\delta}_{n})_{n=0}^{N_{\delta}}$ constructed recursively by \eqref{eq:SFB}. For $n\in\{0,1,\ldots,N_{\delta}-1\}$ and $t\in[t_{n},t_{n+1})$, consider the piecewise linear interpolation 
    \begin{equation}
        X^{\delta}(t)=X^{\delta}_{n}+\frac{t-t_{n}}{\Delta t_{n}^{\delta}}(X^{\delta}_{n+1}-X^{\delta}_{n}).
    \end{equation}
    Then $X^{\delta}$ is a stochastic process with continuous sample paths. Let $X\in\bS^{p}_{d}[0,T]$ denote the unique solution to \eqref{eq:SDIMMO} with initial condition $X(0)=\xi$. Then, 
    \begin{equation}
        \lim_{\delta\to 0}\Ex\left(\sup_{t\in[0,T]}\norm{X^{\delta}(t)-X(t)}^{2}\right)=0.
    \end{equation}
\end{theorem}
 
% Performing a mixed explicit-implicit time discretization of the dynamical system \eqref{eq:SDIMMO}, we arrive that the random iteration 
%\begin{equation}\label{eq:scheme}
%-\frac{X_{n+1}-X_{n}}{\lambda_{n}}\in \opA(X_{n+1})+\beta_{n}\nabla \Psi(X_{n})+\sigma_{n}(X_{n})\xi_{n+1}
%\end{equation}
%where $(\xi_{n})_{n}$ is an iid noise term with zero mean and unit variance. This provides the stochastic forward-backward algorithm 
%\begin{equation}\label{eq:SFB}
%X_{n+1}=(\Id+\lambda_{n}\opA)^{-1}(X_{n}-\lambda_{n}\beta_{n}(\nabla\Psi(X_{n})+\sigma_{n}(X_{n})\xi_{n+1})).
%\end{equation}

\subsection{Complexity analysis of the discrete-time scheme}
Our aim in this section is to derive complexity results for the discrete-time scheme \eqref{eq:SFB} for fixed $T>0$ and vanishing step size $\lambda_{n}=\Delta t_{n}^{\delta}$.  This means that the time discretization vanishes, a setting that is formally not covered by the approximation result  reported in Theorem \ref{th:BerConvergence}.  Additionally, we formally set $N\equiv N_{\delta},X_{n}\equiv X^{\delta}_{n},\beta_{n}\equiv \beta(t_{n})$, and
$$
\sigma_{n}(X_{n})\eqdef \sigma(t_{n},X_{n}),\; \Delta W_{n+1}\eqdef W_{t_{n+1}}-W_{t_{n}} \text{ for }n=0,\ldots,N-1.
$$
Hence, $\Delta W_{n+1}$ is a normally distributed random variable with mean zero and covariance matrix $\lambda_{n}\Id$ for all $n=0,1,\ldots,N-1$. Resulting from this discretization, we define $\Delta B_{n+1}\eqdef \sigma_{n}(X_{n})\Delta W_{n+1}$,  and the discrete-time martingale transform $(B_{n})_{n=0}^{N}$ as the process 
$$
B_{0}\eqdef 0, \quad  B_{n+1}\eqdef B_{0}+\sum_{i=0}^{n}\sigma_{i}(X_{i})\Delta W_{i+1}=\sum_{i=0}^{n}\Delta B_{i+1}.
$$
Adopting this notation, we can realize the Markov chain $(X_{n})_{n=0}^{N}$ as the stochastic recursion
\begin{equation}
    \label{eq:scheme}
\left\{
\begin{aligned}
    &V_{n+1}=\beta_{n}\nabla\Psi(X_{n})+\Delta B_{n+1}, \\
&X_{n+1}=(\Id+\lambda_{n}\opA)^{-1}(X_{n}-\lambda_{n}V_{n+1}). 
\end{aligned}
\right.
\end{equation}

\begin{remark}
Our analysis of the scheme \eqref{eq:scheme} extends \cite{AttCzarPey11} to the stochastic setting, assuming the availability of an unbiased stochastic oracle with decreasing variance. This can be achieved via variance reduction techniques. 
\end{remark}

\begin{remark}
    The discrete-time scheme \eqref{eq:scheme} is not intended to represent a converging numerical discretization of the continuous-time method. Instead, we take the continuous-time system as an inspiration to design numerical algorithms for solving the constrained variational inequality. 
\end{remark}
\begin{remark}
Our scheme can also be used as a numerical method for solving the hierarchical equilibrium problem \eqref{eq:MI} for $C=\argmin_{x\in\R^{d}}\Psi(x)$, under the assumption that we only have access to the lower level optimization problem in terms of a stochastic first-order oracle.  
\end{remark}

\begin{lemma}\label{lem:discrete1}
    Take $z\in\dom(\opA+\NC_{C})$ and $v\in\opA(z)$. Then, for all $\eta> 0$ and \SQ{$n\geq 0$}, we have 
    \begin{equation}\label{eq:discreteenergy1}
    \begin{split}
                \norm{X_{n}-z}^{2}-&\norm{X_{n+1}-z}^{2}-\frac{\eta}{2(1+\eta)}\norm{X_{n+1}-X_{n}}^{2}\geq \frac{2\eta\lambda_{n}\beta_{n}}{1+\eta}\Psi(X_{n})\\
&-\lambda_{n}\beta_{n}\left((1+\eta)\lambda_{n}\beta_{n}-\frac{2}{(1+\eta)L_{\Psi}}\right)\norm{\nabla\Psi(X_{n})}^{2}+2\lambda_{n}\inner{\Delta B_{n+1}+v,X_{n}-z}\\
&-\frac{4\lambda^{2}_{n}(1+\eta)}{\eta}\norm{v}^{2}-\frac{4\lambda^{2}_{n}(1+\eta)}{\eta}\norm{\Delta B_{n+1}}^{2}.
        \end{split}\end{equation}
\end{lemma}

\begin{proof}
    Let $(z,z^{*})\in\gr(\opA+\NC_{C})$ be arbitrary, so that $z^{*}=v+p$ with $v\in\opA(z)$ and $p\in\NC_{C}(z)$. 

In every iteration $n$, we thus have 
$$
\frac{X_{n}-X_{n+1}}{\lambda_{n}}-V_{n+1}\in\opA(X_{n+1}), \text{ and } z^{*}-p=v\in\opA(z).
$$
Monotonicity of $\opA$ gives 
\begin{align*}
2\inner{X_{n}-X_{n+1},X_{n+1}-z}&\geq 2\lambda_{n}\inner{V_{n+1},X_{n+1}-z}-2\lambda_{n}\inner{p-z^{*},X_{n+1}-z}\\
&=2\lambda_{n}\beta_{n}\inner{\nabla\Psi(X_{n}),X_{n+1}-z}+2\lambda_{n}\inner{\Delta B_{n+1},X_{n+1}-z}-2\lambda_{n}\inner{p-z^{*},X_{n+1}-z}.
\end{align*}
Using the 3-point identity 
$$
2\inner{X_{n}-X_{n+1},z-X_{n+1}}=\norm{X_{n+1}-X_{n}}^{2}-\norm{X_{n}-z}^{2}+\norm{z-X_{n+1}}^{2},
$$
we arrive at 
\begin{align*}
\norm{X_{n}-z}^{2}-&\norm{X_{n+1}-z}^{2}-\norm{X_{n+1}-X_{n}}^{2}\geq 2\lambda_{n}\beta_{n}\inner{\nabla\Psi(X_{n}),X_{n}-z} +2\lambda_{n}\inner{\Delta B_{n+1},X_{n}-z}\\
&+2\lambda_{n}\inner{v,X_{n}-z}+2\lambda_{n}\inner{\beta_{n}\nabla\Psi(X_{n})+v+\Delta B_{n+1},X_{n+1}-X_{n}}.
\end{align*}

Since $\Psi\in\bC^{1,1}_{L_{\Psi}}(\R^{d})$ and convex, we have
\begin{align*}
&\inner{\nabla\Psi(X_{n}),X_{n}-z}\geq \frac{1}{L_{\Psi}}\norm{\nabla\Psi(X_{n})}^{2},\text{ and }\\
&\inner{\nabla\Psi(X_{n}),X_{n}-z}\geq \Psi(X_{n}).
\end{align*}
Taking a convex combination of these two estimates yields for every $\eta>0$ and $n\geq 0$, 
\begin{equation}\label{eq:first}
\inner{\nabla\Psi(X_{n}),X_{n}-z}\geq \frac{\eta}{1+\eta}\Psi(X_{n})+\frac{1}{(1+\eta)L_{\Psi}}\norm{\nabla\Psi(X_{n})}^{2}.
\end{equation}
Second, we make use of the identity 
\begin{align*}
\frac{1}{1+\eta}\norm{X_{n+1}-X_{n}+(1+\eta)\lambda_{n}\beta_{n}\nabla\Psi(X_{n})}^{2}&=\frac{1}{1+\eta}\norm{X_{n+1}-X_{n}}^{2}+(1+\eta)\lambda^{2}_{n}\beta^{2}_{n}\norm{\nabla\Psi(X_{n})}^{2}\\
&+2\lambda_{n}\beta_{n}\inner{\nabla\Psi(X_{n}),X_{n+1}-X_{n}},
\end{align*}
in order to bound
\begin{equation}\label{eq:second}
2\lambda_{n}\beta_{n}\inner{\nabla\Psi(X_{n}),X_{n+1}-X_{n}}\geq -\frac{1}{1+\eta}\norm{X_{n+1}-X_{n}}^{2}-(1+\eta)\lambda^{2}_{n}\beta^{2}_{n}\norm{\nabla\Psi(X_{n})}^{2}.
\end{equation}
Combining \eqref{eq:first} and \eqref{eq:second}, together with Young's inequality, we arrive at the estimate
\begin{align*}
\norm{X_{n}-z}^{2}-&\norm{X_{n+1}-z}^{2}-\norm{X_{n+1}-X_{n}}^{2}\geq\frac{2\eta\lambda_{n}\beta_{n}}{1+\eta}\Psi(X_{n})+\left(\frac{2\lambda_{n}\beta_{n}}{(1+\eta)L_{\Psi}}-(1-\eta)\lambda_{n}^{2}\beta_{n}^{2}\right)\norm{\nabla\Psi(X_{n})}^{2}\\
&-\frac{\lambda_{n}}{c}\norm{\Delta B_{n+1}+v}^{2}-\left(\frac{1}{1+\eta}+\lambda_{n}c\right)\norm{X_{n+1}-X_{n}}^{2}+2\lambda_{n}\inner{v,X_{n}-z}+2\lambda_{n}\inner{\Delta B_{n+1},X_{n}-z},
\end{align*}
for every $c>0$. Rearranging and using a square bound, this leads to 
\begin{align*}
\norm{X_{n}-z}^{2}-&\norm{X_{n+1}-z}^{2}-\left(1-\frac{1}{1+\eta}-\lambda_{n}c\right)\norm{X_{n+1}-X_{n}}^{2}\geq \frac{2\eta\lambda_{n}\beta_{n}}{1+\eta}\Psi(X_{n})\\
&-\lambda_{n}\beta_{n}\left((1+\eta)\lambda_{n}\beta_{n}-\frac{2}{(1+\eta)L_{\Psi}}\right)\norm{\nabla\Psi(X_{n})}^{2}+2\lambda_{n}\inner{v,X_{n}-z}+2\lambda_{n}\inner{\Delta B_{n+1},X_{n}-z}\\
&-\frac{2\lambda_{n}}{c}\norm{v}^{2}-\frac{2\lambda_{n}}{c}\norm{\Delta B_{n+1}}^{2}.
\end{align*}
Choosing $c=\frac{\eta}{2\lambda_{n}(1+\eta)}$, we finally arrive at the desired inequality 
\begin{align*}
\norm{X_{n}-z}^{2}-&\norm{X_{n+1}-z}^{2}-\frac{\eta}{2(1+\eta)}\norm{X_{n+1}-X_{n}}^{2}\geq \frac{2\eta\lambda_{n}\beta_{n}}{1+\eta}\Psi(X_{n})\\
&-\lambda_{n}\beta_{n}\left((1+\eta)\lambda_{n}\beta_{n}-\frac{2}{(1+\eta)L_{\Psi}}\right)\norm{\nabla\Psi(X_{n})}^{2}+2\lambda_{n}\inner{v,X_{n}-z}+2\lambda_{n}\inner{\Delta B_{n+1},X_{n}-z}\\
&-\frac{4\lambda^{2}_{n}(1+\eta)}{\eta}\norm{v}^{2}-\frac{4\lambda^{2}_{n}(1+\eta)}{\eta}\norm{\Delta B_{n+1}}^{2}.
\end{align*}
\end{proof}

\begin{lemma}\label{lem:discrete2}
    Assume that $\limsup_{n\to\infty}\lambda_{n}\beta_{n}<\frac{2}{L_{\Psi}}$. Then there exist $c_{1},c_{2}>0$ and $N\in\N$ deterministic, such that for all $n\geq N$ and any $z\in\dom(\opA+\NC_{C})$ and any $v\in\opA(z)$, we have 
    \begin{align*}
        \norm{X_{n+1}-z}^{2}&-\norm{X_{n}-z}^{2}+c_{1}\left(\frac{1}{2}\norm{X_{n+1}-X_{n}}^{2}+2\lambda_{n}\beta_{n}\Psi(X_{n})+\lambda_{n}\beta_{n}\norm{\nabla\Psi(X_{n})}^{2}\right)\\
        &\leq 2\lambda_{n}\inner{\Delta B_{n+1}+v,z-X_{n}}+c_{2}\lambda^{2}_{n}\left(\norm{v}^{2}+\norm{\Delta B_{n+1}}^{2}\right). 
    \end{align*}
\end{lemma}
\begin{proof}
Continuing from the estimate reported in Lemma \ref{lem:discrete1}, we add the term $\frac{\eta}{1+\eta}\lambda_{n}\beta_{n}\norm{\nabla\Psi(X_{n})}^{2}$ to both sides, to obtain 
\begin{align*}
    \norm{X_{n+1}-z}^{2}-&\norm{X_{n}-z}^{2}+\frac{\eta}{2(1+\eta)}\norm{X_{n+1}-X_{n}}^{2}+\frac{2\eta\lambda_{n}\beta_{n}}{1+\eta}\Psi(X_{n})+\frac{\eta}{1+\eta}\lambda_{n}\beta_{n}\norm{\nabla\Psi(X_{n})}^{2}\\
    &\leq \lambda_{n}\beta_{n}\left((1+\eta)\lambda_{n}\beta_{n}-\frac{2}{(1+\eta)L_{\Psi}}+\frac{\eta}{1+\eta}\right)\norm{\nabla\Psi(X_{n})}^{2}+2\lambda_{n}\inner{\Delta B_{n+1}+v,z-X_{n}}\\
    &+\frac{4\lambda^{2}_{n}(1+\eta)}{\eta}\left(\norm{v}^{2}+\norm{\Delta B_{n+1}}^{2}\right).
\end{align*}
As we assume $\limsup_{n\to\infty}\lambda_{n}\beta_{n}<\frac{2}{L_{\Psi}}$, there exists $N\in\N$ such that $\lambda_{n}\beta_{n}<2/L_{\Psi}$ for all $n\geq N$. Moreover, 
$$
\lim_{\eta\to0}\left((1+\eta)\lambda_{n}\beta_{n}-\frac{2}{(1+\eta)L_{\Psi}}+\frac{\eta}{1+\eta}\right)=\lambda_{n}\beta_{n}-\frac{2}{L_{\Psi}}<0,
$$
we can choose $\eta_{0}>0$ sufficiently small, and set 
$$
c_{1}=\frac{\eta_{0}}{1+\eta_{0}},\quad c_{2}=\frac{4(1+\eta_{0})}{\eta_{0}},
$$
so that for all $n\geq N$, 
\begin{align*}
    \norm{X_{n+1}-z}^{2}-&\norm{X_{n}-z}^{2}+c_{1}\left(\frac{1}{2}\norm{X_{n+1}-X_{n}}^{2}+2\lambda_{n}\beta_{n}\Psi(X_{n})+\lambda_{n}\beta_{n}\norm{\nabla\Psi(X_{n})}^{2}\right)\\
    &\leq 2\lambda_{n}\inner{\Delta B_{n+1}+v,z-X_{n}}+c_{2}\lambda_{n}^{2}\left(\norm{v}^{2}+\norm{\Delta B_{n+1}}^{2}\right),     
\end{align*}
as claimed.
\end{proof}
We are now ready to present the main convergence proof of the discrete-time scheme. We impose assumptions which are the natural discrete-time formulations of the continuous-time scheme. 
\begin{theorem}
    Let $(X_{n})_{n\in\N}$ be the non-autonomous Markov chain generated by \eqref{eq:scheme} and let be $z\in\zer(\opA+\NC_{C})$. Assume that $\limsup_{n\to\infty}\lambda_{n}\beta_{n}<\frac{2}{L_{\Psi}}$, and the following conditions hold true: 
    \begin{align}
%&(\forall p\in\range(\NC_{C})):\; \sum_{n\geq 1}\lambda_{n}\beta_{n}h_{C}(p/\beta_{n})<\infty,\\ 
&\revise{\sum_{n\geq 1}\lambda_n\beta_n h_C(p/\beta_n)<\infty\quad \text{for all } p\in\range(\NC_C),}\\
& %\sup_{n\geq 1}
\sup_{x\in\R^{d}}\norm{\sigma_{n}(x)}^{2}_{F}\leq s^{2}_{n},\quad \text{with }\sum_{n=1}^{\infty}\lambda^{2}_{n}<\infty,\sum_{n\geq 1}\lambda_{n}s_{n}^{2}<\infty. 
\end{align}
Then, 
\begin{itemize}
    \item $\lim_{n\to\infty}\Ex[\norm{X_{n}-z}^{2}]$ exists in $\R$. 
    \item $\Ex[\norm{X_{n+1}-X_{n}}^{2}],\Ex[\lambda_{n}\beta_{n}\Psi(X_{n})]$, and $\Ex[\lambda_{n}\beta_{n}\norm{\nabla\Psi(X_{n})}^{2}]$ all converge to 0. 
    \item If $\inf_{n\ge 1}\lambda_{n}\beta_{n}>0$, then $\Psi(X_{n})\to 0$ and $\norm{\nabla\Psi(X_{n})}\to 0 $ \SQ{$\Pr$-}a.s. 
\end{itemize}
\end{theorem}
\begin{proof}
For $(u,z)\in\gr(\opA+\NC_{C})$ define 
\begin{equation}
    \scrE_{n}(z)\eqdef \norm{X_{n+1}-z}^{2}-\norm{X_{n}-z}^{2}+c_{1}\left(\frac{1}{2}\norm{X_{n+1}-X_{n}}^{2}+\lambda_{n}\beta_{n}\Psi(X_{n})+\lambda_{n}\beta_{n}\norm{\nabla\Psi(X_{n})}^{2}\right),
\end{equation}
where $c_1$ is given by Lemma \ref{lem:discrete2}. Since $u\in\opA(z)+\NC_{C}(z)$, there are $v\in\opA(z)$ and $p\in\NC_{C}(z)$ such that $u=v+p.$ A simple computation shows that 
\begin{align*}
    2\lambda_{n}\inner{v,z-X_{n}}-c_{1}\lambda_{n}\beta_{n}\Psi(X_{n})&=2\lambda_{n}\inner{u,z-X_{n}}-c_{1}\lambda_{n}\beta_{n}\Psi(X_{n})
    -2\lambda_{n}\inner{p,z-X_{n}}\\
    &=2\lambda_{n}\inner{u,z-X_{n}}+c_{1}\lambda_{n}\beta_{n}\left(\inner{\frac{2p}{c_{1}\beta_{n}},X_{n}}-\support_{C}\left(\frac{2p}{c_{1}\beta_{n}}\right)-\Psi(X_{n})\right)\\
    &\leq 2\lambda_{n}\inner{u,z-X_{n}}+c_{1}\lambda_{n}\beta_{n}\left(\Psi^{\ast}\left(\frac{2p}{c_{1}\beta_{n}}\right)-\support_{C}\left(\frac{2p}{c_{1}\beta_{n}}\right)\right).
\end{align*}
Then, 
\begin{align*}
\scrE_{n}(z)+c_{1}\lambda_{n}\beta_{n}\Psi(X_{n})\leq 2\lambda_{n}\inner{\Delta B_{n+1},z-X_{n}}+2\lambda_{n}\inner{v,z-X_{n}}+c_{2}\lambda_{n}^{2}\left(\norm{v}^{2}+\norm{\Delta B_{n+1}}^{2}\right),
\end{align*}
which, when combined with the previous display, allows us to conclude that
\begin{align*}
    \scrE_{n}(z)&\leq c_{2}\lambda_{n}^{2}\left(\norm{v}^{2}+\norm{\Delta B_{n+1}}^{2}\right)+2\lambda_{n}\inner{\Delta B_{n+1},z-X_{n}}\\
    &+2\lambda_{n}\inner{u,z-X_{n}}+c_{1}\lambda_{n}\beta_{n}\left[\Psi^{\ast}\left(\frac{2p}{c_{1}\beta_{n}}\right)-\support_{C}\left(\frac{2p}{c_{1}\beta_{n}}\right)\right].
\end{align*}
Setting $u=0$, so that $z\in\zer(\opA+\NC_{C})$, and taking expectations on both sides, we are left with 
$$
\Ex[\scrE_{n}(z)]\leq c_{2}\lambda^{2}_{n}\left(\norm{v}^{2}+\lambda_{n}s^{2}_{n}\right)+c_{1}\lambda_{n}\beta_{n}h_{C}\left(\frac{2p}{c_{1}\beta_{n}}\right).
$$
Upon defining the numerical sequences 
\begin{align}
 & \ca_{n}\eqdef \Ex[\norm{X_{n}-z}^{2}],\\ 
   & \cb_{n}\eqdef c_{1}\Ex[\frac{1}{2}\norm{X_{n+1}-X_{n}}^{2}+\lambda_{n}\beta_{n}\Psi(X_{n})+\lambda_{n}\beta_{n}\norm{\nabla\Psi(X_{n})}^{2}],\\
&\ce_{n}\eqdef c_{2}\lambda^{2}_{n}\left(\norm{v}^{2}+\lambda_{n}s^{2}_{n}\right)+c_{1}\lambda_{n}\beta_{n}h_{C}\left(\frac{2p}{c_{1}\beta_{n}}\right),
\end{align}
we arrive at the recursion 
$$
\ca_{n+1}-\ca_{n}+\cb_{n}\leq\ce_{n} 
$$
with $\ca_{n}\geq 0$ and $\ce_{n}\in\ell^{1}_{+}(\N)$. It follows $\lim_{n\to\infty}\ca_{n}\in\R$ and $(\cb_{n})_{n}\in\ell^{1}_{+}(\N)$. 
\end{proof}

\section{Numerical Experiments}\label{sec:numerics}
%----------------------------------------------------------------------
%%%Numerics	
%----------------------------------------------------------------------
%!TEX root = ./MultiscaleStochastic.tex
%
In this section we report results from some numerical experiments 
conducted in order to illustrate the performance of the numerical scheme 
derived from our continuous time method. The numerical examples we consider are all 
hierarchical optimization problems of the form 
\begin{equation}\label{eq:Test}
\min_{x}\varphi(x) \quad \text{s.t.: }x\in\argmin\{\Psi(u)\vert u\in\R^{d}\}. 
\end{equation}
The defining functions $\varphi:\R^{d}\to\R$ and $\Psi:\R^{d}\to\R$ are convex, and 
$C=\zer(\nabla\Psi)$. We do not enforce $C=\Psi^{-1}(0)$ in our experiments.

\paragraph{The basis pursuit}
This is a classical problem in signal processing \cite{BasisPursuit}, defined by the data $\varphi(x)=\norm{x}_{1}$, and 
$\Psi(x)=\frac{1}{2}\norm{Ax-b}^{2}_{2}$. We compute $\nabla\Psi(x)=A^{\top}(Ax-y)=\sum_{j=1}^{m}A_{j}(Ax-y)_{j}$. 
In many cases, we only have access to subsamples of the gradient of the penalty function. We assume that 
subsampling results in i.i.d subsets $\scrB\subset[m]$, so that at every iteration of the discrete-time scheme, 
we observe the sketched gradient 
$$
\sum_{j\in\scrB_{k}}a_{j}(AX_{n}-y)_{j}+\sigma_{n}\xi_{n+1}\quad \xi_{n+1}\sim N_{d}(0,\mathbf{I}_{d}). 
$$
In the first experiment, we have random generated a matrix with $m=40$ and $d=100$. The stochastic gradient estimator of the penalty 
function is constructed by drawing iid batches of size 4 uniformly at random. 
The evolution of the signal reconstruction process is depicted in Figure \ref{fig:BPD}.
\begin{figure}[h!]
\centering
\includegraphics[width=\textwidth]{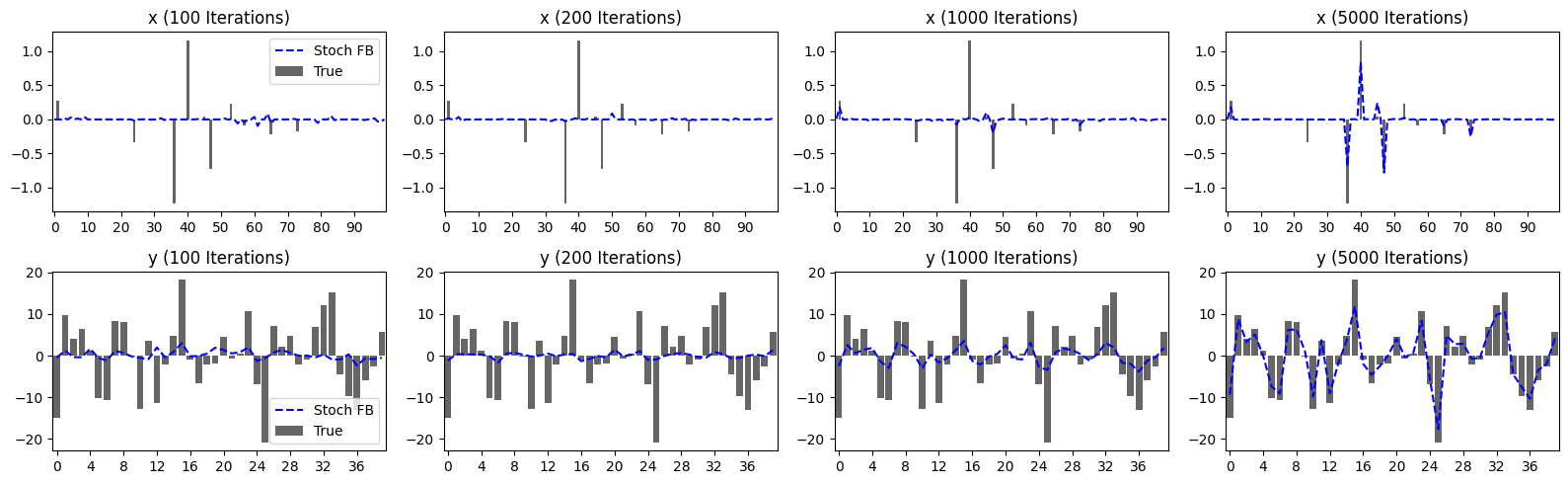}
\caption{Reconstruction of the ground truth based on \eqref{eq:scheme} (top panel) and reconstruction of the signal (bottom panel), after $100,200,1000$ and $5000$ iterations.}
\label{fig:BPD}
\end{figure}
For the conducted numerical experiments, we have taken batch sizes $\abs{\scrB_{k}}=4$. We choose the penalty parameter $\beta_{k}=\frac{(k+10)^{0.75}}{L_{\Psi}}$ and $\lambda_{k}\beta_{k}=\frac{1}{L_{\Psi}},$ where $L_{\Psi}=\norm{A}_{\rm F}$. 

\paragraph{Linear inverse problem with simulated data}
Let $\bar{x}\in\R^{d}$ be a given target vector defined as $\hat{x}=50\1_{d}$. Let $J<d$. Define $\varphi(x)=\norm{x-\hat{x}}_{1}$, 
and let $\Psi:\R^{d}\to\R$ be the quadratic function defined by 
$$
\Psi(x)\eqdef\frac{1}{2}x^{\top}Ax+b^{\top}x=\frac{1}{2}(x_{1}-1)^{2}+\frac{1}{2}\sum_{j=2}^{J}(x_{j-1}-x_{j})^{2}.
$$
The feasible set implied by this bilevel problem is 
$C=\{x\in\R^{d}\vert x_{i}=1,1\leq i\leq J\}=\argmin\Psi=\Psi^{-1}(0).$ Note that the penalty function is not strongly convex, but satisfies the Hölderian error bound from Definition \ref{def:EB}, with $\rho=2$. It can then be readily checked that the global solution of the hierarchical minimization problem \eqref{eq:Test} is given by $\argmin_{x\in C}\varphi(x)=\{\sum_{i=1}^{J}e_{i}+50\sum_{i=J+1}^{d}e_{i}\}=\{x^{*}\}$. 

\begin{figure}[h!]
\centering
\includegraphics[width=\textwidth]{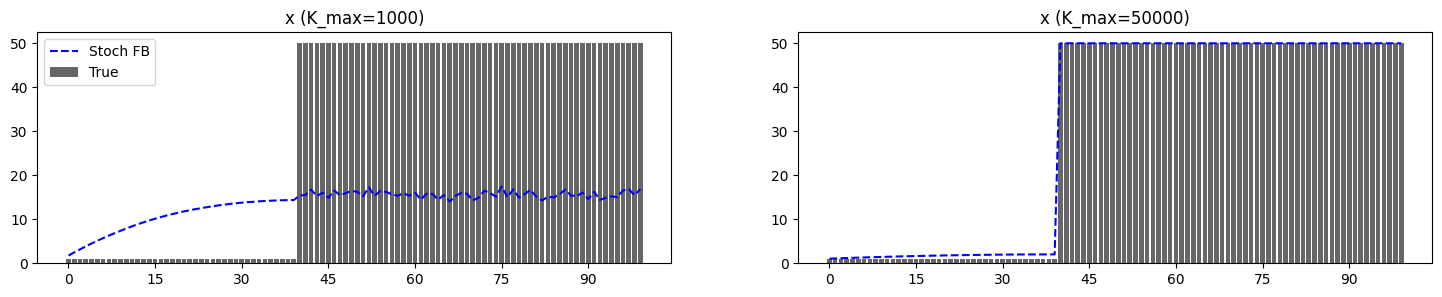}
\caption{Reconstruction of the signal $x^{*}$ after 1000 iterations (left panel) and 50000 iterations (right panel).}
\label{fig:Bilevel}
\end{figure}

\section{Conclusion and perspectives}
\label{sec:conclusion}
%----------------------------------------------------------------------
%%% Conclusion
%----------------------------------------------------------------------
%!TEX root = ./MultiscaleStochastic.tex
%

The paper studies a dynamical system approach to non-smooth convex constrained optimization adopting an exterior penalty formulation. Under the familiar Attouch-Czarnecki condition, we prove asymptotic convergence of the sample paths and the Cesaro mean for maximal monotone inclusions. In the special case where the monotone operator is the subgradient of a proper lower semi-continuous convex function, we derive concrete energy bounds on the objective function values. Asymptotic feasibility of the trajectory is also demonstrated. There are several interesting directions left for future research. 

\paragraph{Invariant measures}
Consider the case $m=d$ and $\sigma(t,x)=\Id$ for all $(t,x)\in[0,\infty)\times\R^{d}$. Additionally, we assume that $\beta(t)=\bar{\beta}$ (constant penalization). Let $b(x)=\bar{\beta}\nabla\Psi(x)$. This setting essentially corresponds to the case in which approximate solutions to the generalized variational inequality \eqref{eq:MI} are requested over finite time. In terms of the stochastic dynamical systems, the assumptions thus made imply that we consider the SDI
\begin{equation}\label{eq:stationary}
\left\{\begin{aligned}
&\dif X(t)+\opA(X(t))\dif t+b(X(t))\dif t\ni \dif W(t) \\ 
&X(0)=x\in\dom(\opA)
\end{aligned}\right.
\end{equation}
 This is the setting of \cite{barbu2005neumann,BarbuPrato08}, where the special case with $\opA=\NC_{K}$ for a closed convex set $K\subset\R^{d}$ with $C^{2}$ boundary $\bd(K)$ and nonempty interior $\Int(K)$ is considered. For this data constellation, the random dynamical system \eqref{eq:stationary} admits a unique strong solution $X^{x}(t)$ with transition semi-group $(P_{t})_{t\geq 0}$, satisfying 
 $$
 (P_{t}f)(x)=\Ex[f(X^{x}(t))]\quad\text{\SQ{for all} } f\in \bC(K),x\in\K,t\geq 0.
 $$
 \cite{BarbuPrato08} characterize the infinitesimal generator of the process as 
 \begin{align*}
 &(\scrG f)(x)=\frac{1}{2}\Delta f(x)-\inner{b(x),\nabla f(x)}\quad f\in\dom(\scrG),x\in K,\\
 &\dom(\scrG)=\{f\in\bC^{2}(\R^{d})\vert \frac{\partial f}{\partial n}=0\text{ on }\bd(K)\},
 \end{align*}
 with $\frac{\partial}{\partial n}$ the outward normal derivative. It is also shown that the semi-group $(P_{t})_{t\geq 0}$ defines an irreducible strong Feller process with a unique invariant measure $\nu$, concentrated on $K$, which is constructively obtained as the weak limit (in the topology of weak convergence on measures) of a Yosida-regularization of the data defining the drift terms $(\opA,b)$. Moreover, assuming that $b:\R^{d}\to\R^{d}$ is a $\mu$-strongly monotone mapping, Proposition 4.2 of \cite{BarbuPrato08} establishes a log-Sobolev inequality on the invariant measure, reading as 
 $$
 \int_{\R^{d}}f^{2}(x)\log(f^{2}(x))\dif \nu(x)\leq \frac{1}{\mu}\int_{\R^{d}}\norm{\nabla f(x)}^{2}\dif\nu(x)+\norm{f}_{L^{2}(\R^{d},\nu)}^{2}\log(\norm{f}^{2}_{L^{2}(\R^{d},\nu)}),
$$
for any continuously differentiable function $f:\R^{d}\to\R$ with bounded domain. Translating these results to our penalty framework, the drift $b=\bar{\beta}\nabla\Psi$ is strongly monotone whenever $\Psi$ is a strongly convex function. This is clearly satisfied for the important case where $\Psi(x)=\frac{1}{2}\norm{Ax-b}^{2}$, with $A^{\top}A$ positive definite. Hence, the invariant measure $\nu$ implicitly depends on the penalty parameter $\bar{\beta}$. Treating now the penalty parameter $\bar{\beta}$ as a hyperparameter, we obtain a sequence of measures $(\nu^{\bar{\beta}})_{\bar{\beta}\geq 0}$, whose asymptotic properties (i.e. when $\bar{\beta}\to\infty$) should capture the limiting invariant distribution of the time-dependent process $X$ analyzed in the main part of this paper. It is an interesting open question to investigate in detail the limiting behavior of the family $(\nu^{\bar{\beta}})_{\bar{\beta}\geq 0}$ as $\bar{\beta}\to\infty.$

\paragraph{Optimal Control for stochastic differential equations with reflections}
\cite{barbu2020optimal} constructs feedback controls for stochastic variational inequalities of the form 
\begin{equation}\label{eq:Control}
\left\{\begin{aligned}
&\dif X(t)+\NC_{K}(X(t))\dif t+b(X(t))\dif t\ni u(t)\dif t +\dif W(t) \\ 
&X(0)=x\in K
\end{aligned}\right.
\end{equation}
where all the data are the same as in \eqref{eq:stationary}, but an additive external input (control) is added to the drift of the dynamics, in terms of a bounded measurable function $u(t)=\pi(X(t))$ (feedback control). Our analysis therefore can be directly adapted to perform a stability analysis of the controlled stochastic process $X$, for a given feedback law. 

%First, it is an important question to obtain more quantitative bounds on the stochastic process. Large deviation type concentration bounds are an important ingredient in uncertainty quantification, and should be established for the penalized dynamical system. Second, it would be interesting to investigate deeper the distributional properties of the process. Bounding the Wasserstein distance under a metric regularity condition, such as the weak sharp minimum, should allow us to derive refined concentration bounds in terms of the laws of the process. Finally, the noise model should be reconsidered. Instead of continuous noise, it would be very interesting to study the dynamics with disruptive jumps (e.g. the sweeping processes), or even hybrid dynamical systems. We leave these important questions for future research. 

\section*{Acknowledgements}

This work was supported by a CSC scholarship and the FMJH Program Gaspard Monge for optimization and operations research and their interactions with data science. M. Staudigl acknowledges support from the Deutsche Forschungsgemeinschaft (DFG) - Projektnummer 556222748 ("Non-stationary hierarchical optimization").

\begin{appendix}

\section{Bounded Variation Functions}
\label{app:BV}
%----------------------------------------------------------------------
%%%Appendix_BVFunctions
%----------------------------------------------------------------------
%!TEX root = ./MultiscaleStochastic.tex
%

We denote by $\bC([a,b];\R^{d})$ the space of continuous functions $f:[a,b]\subset\R\to\R^{d}$. This is a separable Banach space with norm $\norm{f}_{\bC([a,b];\R^{d})}\eqdef\sup_{a\leq t\leq b}\norm{f(t)}$.\\

Let $[a,b]$ be a closed interval from $\R$ and $\scrD_{[a,b]}$ be the set of all partitions $\Delta=\{t_{0}=a,t_{1},\ldots,t_{n}=b\}, n\in\N^{\ast}$. Define the width $\abs{\Delta}=\sup_{0\leq i\leq n-1}\abs{t_{i+1}-t_{i}}$. Given a function $f:[a,b]\to\R^{d}$, we define the variation of $f$ over $[a,b]$ as 
$$
V_{\Delta}(f)=\sum_{i=0}^{n-1}\norm{f(t_{i+1})-f(t_{i})}.
$$
The total variation of $f$ on $[a,b]$ is defined by 
$$
\nu_{a}^{b}(f)\eqdef \sup_{\Delta\in\scrD_{[a,b]}}V_{\Delta}(f)
$$
If $a=0$ and $b=T$, we write $\nu_{T}(f)=\nu_{0}^{T}(f)$. A function $f:[a,b]\to\R^{d}$ has bounded variation on $[a,b]$ if $\nu_{a}^{b}(f)<\infty$. The space of bounded variation functions on $[a,b]$ is denoted by $\BV([a,b];\R^{d})$.

If $x(\cdot)\in\bC([0,T];\R^{d})$ and $k\in\BV([0,T];\R^{d})$, then the Riemann-Stieltjes integral is defined by 
$$
\int_{a}^{b}\inner{x(t),\dif k(t)}=\lim_{\abs{\Delta}\to 0}\sum_{i=0}^{n_{\Delta}-1}\inner{x(t_{i}),k(t_{i+1})-k(t_{i})}. 
$$
Equipped with the norm 
$$
\norm{f}_{\BV([a,b];\R^{d})}\eqdef \norm{f}_{\bC([a,b];\R^{d})}+\nu_{a}^{b}(f) 
$$
the space $\BV([a,b];\R^{d})$ is a Banach space. An element $k\in\BV([a,b];\R^{d})$ can be identified with the continuous linear functional 
$$
x\in\bC([a,b];\R^{d})\mapsto \inner{x(a),k(a)}+\int_{a}^{b}\inner{x(t),\dif k(t)}
$$
%With this identification, we can consider $\BV([a,b];\R^{d})$ as the dual space to $\bC([a,b];\R^{d})$. 
 
\section{Existence and Uniqueness of Solutions}
\label{app:existence}
%----------------------------------------------------------------------
%%% Appendix_Existence
%----------------------------------------------------------------------
%!TEX root = ./MultiscaleStochastic.tex
%

We verify existence and uniqueness of solutions for the stochastic differential inclusion \eqref{eq:SDIMMO}. We achieve this by verifying the assumptions leading to the corresponding existence and uniqueness theorems in \cite{pardoux2014stochastic}. We start with the most restrictive assumption, which essentially precludes evolution equations in infinite dimensional Hilbert spaces in the general maximal monotone case.  

\begin{assumption}\label{ass:interior}
$\scrC\cap\Int(\dom(\opA))\neq\emptyset$. 
\end{assumption}
Note that in case where $\opA=\partial\varphi$ for $\varphi\in\Gamma_{0}(\R^{d})$, the $\Int(\dom(\partial\varphi))=\Int(\dom(\varphi))$. Assumption \ref{ass:interior} implies that for all $u_{0}\in\Int(\dom(\opA))$ there exists an $r_{0}>0$ such that $\B(u_{0},r_{0})\subset\dom(\opA)$ and 
$$
\opA^{\#}_{u_{0},r_{0}}\eqdef \sup\{\norm{x}\;\vert x\in\opA(u_{0}+v),\norm{v}\leq r_{0}\}<\infty.
$$
By Assumption \ref{ass:Psi}, we have the following anti-monotonicity condition 
\begin{equation}\label{eq:AntiMonotone}
\inner{x-y,\nabla\Psi(y)-\nabla\Psi(x)}\leq L_{\Psi}\norm{x-y}^{2}.
\end{equation}
Furthermore, Assumption \ref{ass:Noisebound} is in place in the regime (ASN). Note that the condition \eqref{eq:AntiMonotone} implies for $u_{0}\in\scrC\cap\Int(\dom(\opA))$ the growth condition
$$
\inner{-\beta(t)\Psi(x),x-u_{0}}\leq \beta(t)L_{\Psi}\norm{x-u_{0}}^{2}\equiv \mu(t)\norm{x-u_{0}}^{2}.
$$ 
Proposition 6.19 from \cite{pardoux2014stochastic} yields the a-priori bound for every pair $(X,K)$ satisfying $\dif K(t)\in\opA(X(t))$ $(t,\omega)$-a,e, 
$$
r_{0}\dif\nu_{t}(K)\leq \inner{X_{t}-u_{0},\dif K(t)}+\left(\opA^{\#}_{u_{0},r_{0}}\norm{X(t)-u_{0}}+\opA^{\#}_{u_{0},r_{0}}\right)\dif t
$$
as signed measures. Since $\nabla\Psi(u_{0})=0$, we can rearrange this inequality to become
$$
r_{0}\dif\nu_{t}(K)+\inner{X(t)-u_{0},-\beta(t)\nabla\Psi(X(t))}\dif t\leq \inner{X(t)-u_{0},\dif K(t)}+\left(\opA^{\#}_{u_{0},r_{0}}\norm{X(t)-u_{0}}+\opA^{\#}_{u_{0},r_{0}}\right)\dif t
$$
By the Lipschitz continuity of the variance $\sigma(t,x)$ \eqref{eq:noiseLip}, we have for all $\lambda>1$ 
$$
\norm{\sigma(t,x)}_{\rm F}^{2}\leq \frac{\lambda}{\lambda-1}\norm{\sigma(t,u_{0})}^{2}_{\rm F}+\lambda\ell^{2}(t)\norm{x-u_{0}}^{2}.
$$
Some simple algebra leads us then to the comparison
\begin{align*}
&r_{0}\dif\nu_{t}(K)-\inner{X_{t}-u_{0},\dif K(t)}-\beta(t)\inner{X(t)-u_{0},\nabla\Psi(X(t))}\dif t +\left(\frac{p-1}{2}+\opA^{\#}_{u_{0},r_{0}}\right)\norm{\sigma(t,X(t))}^{2}_{\rm F}\\
&\leq \norm{X(t)-u_{0}}\opA^{\#}_{u_{0},r_{0}}\dif t+\left( \opA^{\#}_{u_{0},r_{0}}+\left(\frac{p-1}{2}+9p\lambda\right)\norm{\sigma(t,u_{0})}^{2}_{\rm F}\right)\dif t+\lambda\left(\frac{p-1}{2}+9p\lambda\right)\ell^{2}(t)\norm{X(t)-u_{0}}^{2}\dif t
\end{align*}
for $p>1$. We define the following processes:
\begin{align*}
&D(t)\eqdef r_{0}\nu(K)_{t},\\ 
&R(t)\eqdef \int_{0}^{t}\left( \opA^{\#}_{u_{0},r_{0}}+\left(\frac{p-1}{2}+9p\lambda\right)\norm{\sigma(t,u_{0})}^{2}_{\rm F}\right)\dif t\\
&N(t)\eqdef \opA^{\#}_{u_{0},r_{0}}t,\\
&V(t)\eqdef \int_{0}^{t}\lambda\left(\frac{p-1}{2}+9p\lambda\right)\ell^{2}(t)\dif t.
\end{align*}
Using this notation, we can rewrite the above comparison of signed measures as 
\begin{align*}
&\dif D(t)-\inner{X_{t}-u_{0},\dif K(t)}-\beta(t)\inner{X(t)-u_{0},\nabla\Psi(X(t))}\dif t +\left(\frac{p-1}{2}+\opA^{\#}_{u_{0},r_{0}}\right)\norm{\sigma(t,X(t))}^{2}_{\rm F}\\
&\leq \norm{X(t)-u_{0}}\dif N(t)+\dif R(t)+\norm{X(t)-u_{0}}^{2}\dif V(t).
\end{align*}
This is the main bound involved in the existence (Theorem 4.11) and uniqueness (Theorem 4.18) statements in \cite{pardoux2014stochastic}, slightly simplified due to the fact that the anchor point $u_{0}$ induces a zero gradient of the penalty function. Our existence and uniqueness theorem therefore follows from the just mentioned results.

\section{Auxiliary facts and omitted proofs}
\label{app:proofs}
%----------------------------------------------------------------------
%%%Appendix:Proofs		
%----------------------------------------------------------------------
%!TEX root = ./MultiscaleStochastic.tex
%

We need the following continuous-time version of the celebrated Robbins-Siegmund lemma. 
\begin{proposition}\label{prop:RS}
Let $(A_{t})_{t\geq 0}$ and $(U_{t})_{t\geq 0}$ be two continuous adapted increasing processes with $A_{0}=U_{0}=0$ a.s. Let $(M_{t})_{t\geq 0}$ be a real-valued continuous local martingale with $M_{0}=0$ a.s. Let $\xi$ be a non-negative $\scrF_{0}$-measurable random variable. Define 
$$
X_{t}=\xi+A_{t}-U_{t}+M_{t}\quad t\geq 0.
$$
If $X_{t}$ is non-negative and $\lim_{t\to\infty}A_{t}<\infty$, then $\lim_{t\to\infty}X_{t}$ exists and is finite, and $\lim_{t\to\infty}U_{t}<\infty$. 
\end{proposition}
The following result is due to \cite{combettes2015stochastic} (Proposition 2.3(iii),(iv)).
\begin{proposition}\label{prop:CP}
Let $\scrS$ be a non-empty closed subset of $\R^{d}$ and $(X_{n})_{n}$ a stochastic process living in $\R^{d}$. Suppose that there exists $\Omega_{0}\subset\Omega$ with $\Pr(\Omega_{0})=1$ and, for every $\omega\in\Omega_{0}$, and every $z\in\scrS$, the process $(\norm{X_{n}(\omega)-z})_{n\in\N}$ converges. If additionally all accumulation points of $(X_{n}(\omega))_{n\in\N},\omega\in\Omega_{0},$ are contained in $\scrS$, then $(X_{n})_{n\in\N}$ converges $\Pr$-a.s. to a $\scrS$-valued random variable.
\end{proposition}
For every locally integrable curve $\bx:\R_{\geq 0}\to\R^{d}$, we define its average 
$$
\Avg(\bx;s,t)\eqdef\frac{1}{t-s}\int_{s}^{t}\bx(r)\dif r\quad \text{for all }t\geq s\geq 0. 
$$
The following Opial-like Lemma is reported as Lemma 2.3 in \cite{attouch2018asymptotic}. Associated with a continuous function $\bx\in\bC([0,\infty);\R^{d})$, its omega limit set is defined as 
$$
\Lim(\bx)\eqdef\{z\in\R^{d}\vert \exists (t_{k})\subseteq(0,\infty)\text{ with }t_{k}\uparrow \infty \text{ and }\bx(t_{k})\to z \text{ as }k\to\infty\}.
$$

\begin{lemma}
    Consider $\bx\in\bC(\R_{\geq 0};\R^{d})$ and $\scrS\neq\emptyset$. Suppose that $\Lim(\bx)\subset\scrS$, and for all $x^{*}\in\scrS$ the limit $\lim_{t\to\infty}\norm{\bx(t)-x^{*}}$ exists. Then there exists $x^{*}\in\scrS$ such that $\lim_{t\to\infty}\Avg(\bx;t_{0},t)=x^{\ast}$. 
\end{lemma}
Applying this result to the realization of a continuous stochastic process $t\mapsto X(t,\omega)$, we obtain immediately a measurable convergence result of the averaged trajectory. Given a stochastic process $X:\R_{\geq 0}\times\Omega\to\R^{d}$ with continuous sample paths $t\mapsto X_{t}(\omega)$, denote by 
$$
\bar{X}^{t_{0}}_{t}(\omega)\eqdef \Avg(X(\bullet,\omega);t_{0},t).
$$

\begin{corollary}\label{cor:CP}
Let $\scrS\neq\emptyset$ and $X:\R_{\geq 0}\times\Omega\to\R^{d}$ be a continuous stochastic process living in $\R^{d}$. Suppose that there exists $\Omega_{0}\subset\Omega$ with $\Pr(\Omega_{0})=1$ and, for every $\omega\in\Omega_{0}$, and every $x^{*}\in\scrS$, the sequence $(\norm{\bar{X}^{t_{0}}_{t}(\omega)-z})_{t\geq t_{0}}$ converges. If additionally all accumulation points of $(\bar X^{t_{0}}_{t}(\omega))_{t\geq t_{0}},\omega\in\Omega_{0},$ are contained in $\scrS$, then there exists an $\scrS$-valued random variable $X_{\infty}\in L^{0}(\Omega,\R^{d})$ such that $\lim_{t\to\infty}\bar{X}^{t_{0}}_{t}(\omega)=X_{\infty}(\omega)$ almost surely. 
\end{corollary}

\paragraph{Proof of Lemma \ref{lem:gap}}

If $x\in B_{\delta}$, then $x\in C\cap\dom(\opA)$ and there exists $x^{*}\in\opA(x)$. We can therefore choose $(y,y^{*})=(x,x^{*})$ in the definition of $\gap_{\delta}(x)$, to conclude $\gap_{\delta}(x)\geq 0$. Moreover, the relation $\Theta_{\delta}(x)\leq\Theta(x)$ gives $\Theta_{\delta}(x)=0$ whenever $x\in B_{\delta}\cap\scrS.$

%  \textcolor{red}{  
%    Assume now that $\Theta_{\delta}(x)=0$ for $x\in C$ with $\norm{x-a}<\delta$. Define 
%\begin{equation}
%g(x)=\sup_{(y,v)\in\gr(\opA)}\inner{v,x-y}.
%\end{equation}
%We first show that for $x\in C$ we have $g(x)=\Theta(x)$. Indeed, by definition of the gap function, we have 
%\begin{align*}
%\Theta(x)&=\sup_{y\in\dom(\opA)\cap C}\sup_{y^{*}\in\opA(y)+\NC_{C}(y)}\inner{y^{*},x-y}\\
%&=\sup_{y\in\dom(\opA)\cap C}\sup\{\inner{v+\xi,x-y}\vert v\in\opA(y),\xi\in\NC_{C}(y)\}\\
%&=\sup_{y\in\dom(\opA)}\{\sup_{v\in\opA(y)}\inner{v,x-y}+\sup_{\xi\in\NC_{C}(y)}\inner{\xi,x-y}\}
%\end{align*}
%If $x\in C$ and $y\in\dom(\opA)\cap C$ fixed, then $\inner{\xi,x-y}\leq 0$ for all $\xi\in\NC_{C}(y)$. Hence, 
%$$
%\Theta(x)\leq\sup_{y\in\dom(\opA)\cap C}\inner{v,x-y}=g(x).
%$$
%Conversely, since $0\in\NC_{C}(y)$, we have 
%$$
%\inner{v,x-y}\leq\sup_{\xi\in\NC_{C}(y)}\inner{v+\xi,x-y}.
%$$
%It follows $g(x)\leq\Theta(x)$, showing equality of the two functions for every $x\in C$.Since $B_{\delta}\subseteq C$, we conclude from the above that 
%$$
%g_{\delta}(x)=\sup_{y\in B_{\delta}}\sup_{v\in\opA(y)}\inner{v,x-y}=\Theta_{\delta}(x)
%$$
%for $x\in C$. }
%
\revise{
 For $x\in C$, define 
\begin{equation}
g_\delta(x)=\sup_{\substack{y\in B_\delta\\ v\in \opA(y)}}\inner{v,x-y}.
\end{equation}
Since $B_\delta\subset C$, for every $y\in B_\delta$ and $x\in C$, we have $\sup_{\xi\in \NC_C(y)}\inner{\xi,x-y}=0$, where we have used the definition of the normal cone and the fact that $0\in \NC_C(y)$. Hence,
$$
\Theta_\delta(x)=\sup_{\substack{y\in B_\delta\\v\in \opA(y),\,\xi\in \NC_C(y)}}\inner{v+\xi,x-y}=\sup_{\substack{y\in B_\delta\\ v\in \opA(y)}}\inner{v,x-y}=g_\delta(x).
$$
}

Assume $\gap_{\delta}(\bar{x})=0$ for some $\bar{x}\in C$ with $\norm{\bar{x}-a}<\delta$. Then, $g_{\delta}(\bar{x})=0$, and consequently 
$$
\sup_{v\in\opA(y)}\inner{v,\bar{x}-y}\leq 0\quad\text{\revise{for all} } y\in B_{\delta}.
$$
We first show that $\bar{x}$ is a solution of the generalized variational inequality with set-valued operator $\opA$ and constraint $B_{\delta}$. For the sake of obtaining a contradiction, let's assume that $\bar{x}$ is not a solution. Then, there exists $z\in B_{\delta}$ such that 
$$
\sup_{u\in\opA(\bar{x})}\inner{u,z-\bar{x}}<0.
$$
For $t\in[0,1]$ define 
\begin{align*}
&x(t)=\bar{x}+t(z-\bar{x}),\text{ and }\nu(t)=\sup_{u\in\opA(x(t))}\inner{u,z-x(t)}.
\end{align*}
Since $\nu(0)<0$, and $\opA$ is USC, there exists $r>0$ sufficiently small, for which $\nu(r)<0$. Hence, 
\begin{align*}
0&<\inf_{u\in\opA(x(r))}\inner{u,\bar{x}-x(r)}\leq\sup_{u\in\opA(x(r))}\inner{u,\bar{x}-x(r)}\leq g_{\delta}(\bar{x})=0.
\end{align*}
A contradiction.\\ 
Next, we show that $\bar{x}$ is a solution to the generalized variational inequality with set-valued operator $\opA:\R^{d}\to 2^{\R^{d}}$ and constraint $C$. Assume not. Then, there exists $z\in C$ with 
$$
\sup_{u\in\opA(\bar{x})}\inner{u,z-\bar{x}}<0.
$$
From the above established fact, we conclude that $z\in C\setminus B_{\delta}$. Since $\norm{\bar{x}-a}<\delta$, there exists $\bar{t}>0$ such that $x(t)=\bar{x}+t(z-\bar{x})\in B_{\delta}$ for all $t\in[0,\bar{t})$. Hence, 
$$
\sup_{u\in\opA(\bar{x})}\inner{u,x(t)-\bar{x}}\geq 0\qquad\forall t\in[0,\bar{t}).
$$
However,  for any $u\in\opA(\bar{x})$ fixed, we have $\inner{u,x(t)-\bar{x}}=t\inner{u,z-\bar{x}}<0.$ A contradiction.     
\qed

\end{appendix}

\bibliographystyle{abbrvnat}
\bibliography{PenaltyDynamics}
\end{document}